\documentclass[a4paper, reqno]{amsart}

\usepackage{titlesec}
\usepackage{amsmath,amsthm,amsfonts}
\usepackage{fullpage}
\usepackage{enumerate}
\usepackage{amssymb}
\usepackage{graphicx}
\usepackage{tikz}
\usepackage{tikz-cd}
\usepackage[top=1.5in,bottom=1.2in,left=1in,right=1in]{geometry}
\usepackage{mathtools}
\usepackage[colorlinks=true, linkcolor=teal, citecolor=blue]{hyperref}
\usepackage{cleveref}
\usepackage{xcolor}
\usepackage{xfrac}
\usepackage[all,cmtip]{xy}
\usepackage{faktor}

\titleformat{\section}
  {\centering\normalfont\scshape\fontsize{12}{17}\bfseries}
  {\thesection}
  {1em}
  {}
\titleformat{\subsection}
  {\normalfont\fontsize{12}{14}\bfseries}
  {\thesubsection}
  {1em}
  {}

\newcommand{\D}[1]{\mathbb{D}^{#1}}

\newcommand{\rp}[1]{\mathbb{R}\mathrm{P}^{#1}}

\newcommand{\Z}{\mathbb{Z}}

\newcommand{\R}{\mathbb{R}}

\newcommand{\Sp}[1]{\mathbb{S}^{#1}}

\newcommand{\wi}[1]{\widetilde{#1}}

\newcommand{\cc}[1]{\mathcal{C}{\left(#1\right)}}
 \newcommand{\csum}{{\#}}
\newcommand{\Ext}{\mbox{Ext}}
\newcommand{\FF}{\mathcal{F}}
\newcommand{\EE}{\mathcal{E}}

\newcommand{\Th}[1]{\Theta_{#1}}
\newcommand{\lara}[1]{\{{#1}\}}
\newcommand{\ra}{\rightarrow}

\newcommand{\lr}{\longrightarrow}

\newcommand{\opl}[1]{\oplus{_{#1}}}
\newcommand{\kr}[1]{\mathrm{Ker}{\left(#1\right)}}
\newcommand{\ckr}[1]{\mathrm{Coker}{\left(#1\right)}}
\newcommand{\sckr}[1]{\mathrm{cok}{\left(#1\right)}}
\newcommand{\im}[1]{\mathrm{Im}{\left(#1\right)}}

\newcommand{\ol}[1]{\overline{#1}}

\newcommand{\id}{\mathrm{Id}}
\newcommand{\btau}{\scalebox{1}[1.2]{$\tau$}}
\newcommand{\bpi}{\scalebox{1}[1.2]{$\pi$}}

\newtheorem{thrm}{Theorem}[section]
\newtheorem{lemm}[thrm]{Lemma}
\newtheorem{cor}[thrm]{Corollary}
\newtheorem{propn}[thrm]{Proposition}

\newtheorem{defn}{Definition}[section]
\newtheorem{rem}[thrm]{Remark}

\numberwithin{equation}{section}

\Crefname{lemm}{Lemma}{Lemmas}
\crefname{lemm}{lemma}{lemmas}
\crefname{thrm}{theorem}{theorems}
\Crefname{thrm}{Theorem}{Theorems}

\newtheorem{thrmnonum}{Theorem}

\title{Smooth structures on PL-manifolds of dimensions between 8 and 10}

\numberwithin{equation}{section}
\begin{document}

 \author{Samik Basu}
\address{Stat-Math Unit, Indian Statistical Institute Kolkata, India, 700108.}
\email{samik.basu2@gmail.com; samikbasu@isical.ac.in}
	
\author{Ramesh Kasilingam}
	
	\address{Department of Mathematics,
		Indian Institute Of Technology Chennai, India, 600036.}
\email{rameshkasilingam.iitb@gmail.com  ; rameshk@iitm.ac.in}

 \author{Priyanka Magar-Sawant}
	
	\address{Department of Mathematics,
		Indian Institute Of Technology Bombay, India, 400076.}
\email{priyanka.ms.math@gmail.com}

	

\maketitle

\begin{abstract}
     In this paper, we identify the concordance classes of  smooth structures on $PL$-manifolds of dimension between $8$ and $10$ in terms of the cohomology and Steenrod operations. This leads to the computation of the homotopy inertia groups. Finally we discuss the special cases of Lens spaces and real projective spaces.
\end{abstract}





\maketitle

\section{Introduction}

The study of smooth manifolds and their piece-wise linear ($PL$)-triangulations is one of the essential and active topics in differential topology. Shortly after Milnor  discovered exotic smooth $7$-spheres\cite{KervaireMilnorGroupsOfHomotopySpheres}, Kervaire constructed a $PL$ 10-manifold without any smooth manifold in its homotopy type, and a new exotic $9$-sphere\cite{Kervaire_1960}. This motivates the problem of classifying all smooth structures on a $PL$-manifold if exists, compatible with the underlying PL-structure, up to some suitable equivalence relation. 
In higher dimensions, the classification of compatible smooth structures up to several equivalence relations has been studied (see \cite{KirbySiebenmannFoundationalEssays,kui,Cairns-Hirsch-Mazur,mun,Sullivan1996}).
In this paper, we consider one of such equivalence relations, called concordance.\\

\noindent\textbf{Convention:} We work in the category of oriented smooth manifolds such that all morphisms are $PL$-smooth,  implicitly assuming all manifolds are closed connected smooth oriented of dimension $\geq 5$, and that all maps are orientation preserving.
\begin{defn} Let $M$ be a closed smooth manifold. Let $(N, f )$ be a pair consisting of a smooth manifold $N$ together with a $PL$-homeomorphism $f: N \lr M $. Two such pairs $(N_1 , f_1 )$ and $(N_2 , f_2 )$ are $PL$-concordant provided there exists a diffeomorphism $g : N_1 \lr N_2$ and a $PL$ homeomorphism
$F : N_1 \times [0, 1] \lr M \times [0, 1]$ such that $F|_{N_1 \times0} = f_1$ and $F |_{N_1 \times 1} = f_2 \circ g$. 
\end{defn}
The set of all such $PL$-concordance classes is denoted by $\cc{M}$. The $PL$-concordance class of $(N, f )$ is denoted by $[N, f ]$, and the class $[M, \id]$ of the identity $\id: M \lr M$ can be considered as the base point of $\cc{M}$.
The study of $\cc{M}$ typically proceeds by reducing to bundle theory and then to homotopy theory. In fact, Cairns-Hirsch-Mazur \cite{Cairns-Hirsch-Mazur} proved that, if $ M $ admits a smooth structure then there is a set bijection 
\begin{equation}\label{cairns}
   \cc{M} \cong [M,PL/O] ,
\end{equation}
where $PL/O$ is an H-space, (actually an infinite loop space) that is a homotopy fiber of the forgetful map $ BO\to BPL$. Note that, the spaces $BO$ and $BPL$ have compatible commutative H-space structures arising from the Whitney sum of bundles \cite[p.92]{KirbySiebenmannFoundationalEssays}. Hence $[M,PL/O]$ has a group structure. The bijection in \eqref{cairns} has some immediate consequences. One consequence is that $\cc{M}$ admits an abelian group structure, with $[M,\id]$ acting as the identity element. Another consequence is the isomorphism between the groups $\cc{\mathbb{S}^n}$ and $\Theta_n$, representing $h$-cobordism classes of smooth homotopy spheres. For $n\geq 5$, the group $\Theta_n$ can also be identified with the set of all (oriented) diffeomorphism classes of smooth homotopy spheres. Explicit calculations of concordance groups $\cc{M}$ have been performed for certain manifolds $M$, including the product of standard spheres $\mathbb{S}^i\times \mathbb{S}^j$, an $\mathbb{S}^j$-bundle over $\mathbb{S}^i$ \cite{mun}, as well as complex and quaternionic projective spaces \cite{RKHomotopyInertiaGroupsAndTangentialStructures,RKInertiaGroupsAndSmoothStructuresOnQuaternionicProjectiveSpaces}. Moreover, through obstruction theory and the fact that $PL/O$ is 6-connected, one can establish that the group $\cc{M}\cong H^7(M;\Th{7})$ for a closed smooth $7$-manifold $M$, where $\Th{7}\cong\sfrac{\Z}{28}$ is the group of homotopy $7$-spheres. In this paper, we extend this result to manifolds of dimensions $n=8, 9, 10$, utilizing the structure of the 10th-Postnikov section of $PL/O$ (see \ref{10th}).  
\begin{thrmnonum}\label{thm A} 
 The smooth concordance structure set $\mathcal{C}(M)$ for manifolds $M$ with $8 \leq \dim(M) \leq 10$ is explicitly determined in terms of the action of Steenrod operations on the cohomology of $M$.
\end{thrmnonum}
A detailed discussion of these results is presented in \Cref{thm:8-manifold decomposition}, \Cref{9-mfld pl decomposition thm}, and \Cref{pl decomposition of 10 mfld}. The underlying idea is to utilize cohomology operations to gain sufficient knowledge of the cell attachments of $M$ in degrees $7$ through $10$, enabling the computation of $[M,PL/O]$ through the initial stages of the Postnikov tower of $PL/O$.


Recall that the group $\Th{n}\cong \cc{\Sp{n}}$ acts on the smooth structure set $S^{\text{Diff}}(M)$ \cite{Browder1972Surgery,Wall}, given by 
\begin{equation}\label{eq: Theta action}
    \begin{split}
    \Th{n}\times S^{\text{Diff}}(M) & \lr S^{\text{Diff}}(M) \\
    {([\Sigma,f],[N,g])} & \mapsto {[\Sigma\# N,f\# g]}
\end{split}
\end{equation}

The stabilizer of this action at the base point $[M,\id]$ is known as the homotopy inertia group, denoted by $I_h(M)$. 
It follows from \cite{KervaireMilnorGroupsOfHomotopySpheres} that, for dimensions $8\leq n\leq 10$, the group $\Th{n}$ fits into the following split short exact sequence
\[0\lr bP_{n+1}\lr \Th{n} \lr \faktor{\pi_n^s}{\im{J}}\lr 0,\]
where $\faktor{\pi_n^s}{\im{J}}\subseteq \pi_n(G/O)$. Note that, $\faktor{\pi_8^s}{\im{J}}=\sfrac{\Z}{2}\lara{\epsilon}$, $\faktor{\pi_9^s}{\im{J}}=\sfrac{\Z}{2}\lara{\mu}\oplus \sfrac{\Z}{2}\lara{\eta\circ\epsilon}$, and $\faktor{\pi_{10}^s}{\im{J}}=\sfrac{\Z}{2}\lara{\eta\circ\mu}\oplus \sfrac{\Z}{3}\lara{\beta_1}$.
In this paper, we prove the following by using the structure of $\cc{M}$ given in \Cref{thm A} together with the Postnikov section of $PL/O$:
\begin{thrmnonum}
    Let $M$ be  a closed oriented smooth $n$-manifold for $8\leq n\leq 10$. 
    Then the stabilizer of the action of $~\faktor{\Th{n}}{bP_{n+1}}$ given in \eqref{eq: Theta action} on the base point $[M,\id]$ is explicitly determined in terms of Steenrod operations on the cohomology of $M$ (see \Cref{thm: Ih(M8)}, \Cref{thm: Ih(9)}, and \Cref{thm: 10 Ih}).
\end{thrmnonum}

Recall the lens space $L^{2n+1}(m) = \Sp{2n+1}/\Z_{m}$, where the group action is given by $(z_0,z_1,\dots,z_n)\mapsto(\alpha z_0,\alpha z_1,\dots,\alpha z_n)$ with $\alpha=\text{exp}^{\frac{2\pi i}{m}}$. The following theorem yields the computation of the inertia group of $L^9(m)$ and $\rp{n}$ for $n=8$ and $10$.
\begin{thrmnonum}\label{thm: C}
       Let $m$ be a positive integer and $n$ be a non-negative integer. 
       \begin{enumerate}[(i)]
        \item Let $m=2n+1$. Then for any exotic sphere $\Sigma\in \Th{9}$,
         the connected sum $L^9(m)\#\Sigma$ is not diffeomorphic to $L^9(m)$.
        
        \item Let $m=4n+2$. Then there is a unique exotic sphere $\Sigma\in \Th{9}$ such that $L^9(m)\csum\Sigma$ is diffeomorphic to $L^9(m)$.

        \item Let $m=4n$. Then, there are exactly four exotic spheres $\Sigma_1,\Sigma_2,\Sigma_3,\Sigma_4\in \Th{9}$ such that no two of the manifolds $L^9(m)$, $L^9(m)\csum\Sigma_1$, $L^9(m)\csum\Sigma_2$, $L^9(m)\csum\Sigma_3$, and $L^9(m)\csum\Sigma_4$ are diffeomorphic.

        \item Let $\Sigma\in\Th{8}$ be the exotic sphere. Then $\rp{8}\csum\Sigma$ is not diffeomorphic to $\rp{8}$.

        
         \item   For any homotopy $10$-sphere $\Sigma\in \Th{10}$, the connected sum $\rp{10}\csum\Sigma$ is diffeomorphic to $\rp{10}$.
        
    \end{enumerate}
\end{thrmnonum}
\Cref{thm: C} $(i)$, $(ii)$, and $(iii)$ are immediate consequences of \Cref{thm: Ih(9)} and \cite[Theorem 4.2]{KervaireObstruction}. The last two assertions of Theorem \ref{thm: C} will be proved in \Cref{thm: Ih(M8)} and \Cref{thm: rp10}.

 
 \subsection{Notation} 
 \begin{itemize}
 \item Let $O_n$ be the orthogonal group, $PL_n\subset O_n$ is the group of piece-wise linear homeomorphisms, and $G_n$ be the set of homotopy equivalences. Denote by $O=\underset{n\ra\infty}{colim}~O_n$, $PL= \underset{n\ra\infty}{colim}~ PL_n$, and $G=\underset{n\ra\infty}{colim}~ G_n$ \cite{KuiperLashofMicrobundles,LashofMicrobundlesandsmoothing}. 
 \item Let $G/O$ be the homotopy fiber of the canonical map $BO\to BG$ between the classifying spaces for stable vector bundles and stable spherical fibrations \cite[\S2 and \S3]{MilgramTheClassifyingSpacesForSurgeryAndCobordismOfManifolds}, and $G/PL$ be the homotopy fiber of the canonical map $BPL\to BG$ between the classifying spaces for $PL$ $\mathbb{\R}^n$-bundles and stable spherical fibrations \cite{YuliRudyak}.
 \item For an infinite loop space $X$ we use the small letter $x$ to denote the connective spectrum such that $\Omega^\infty (x) \simeq X$. We use this notation to define the spectra $g$, $o$, $pl$, $pl/o$, $g/o$, $g/pl$. 
 \item The Eilenberg MacLane spectrum for an Abelian group $A$ is denoted by $HA$. 
 \item The notation $\{ -, - \}$ is used to denote the stable homotopy classes of maps between spectra. 
 \item The notation $\btau_{\leq m}$ is reserved for the $m^{th}$ Postnikov section. It satisfies $\bpi_i \btau_{\leq m}(E) = \bpi_i(E)$ for $i\leq m$ and $0$ if $i>m$. The notation ${}^{>m}\btau$ refers to the $m$-connected cover, which is also the fiber of $X \to \btau_{\leq m} X$. 
 \item The Moore space $M(A,n)$ is the space whose reduced homology is concentrated in degree $n$, whence it is isomorphic to $A$.

 \item The notations used for the generators of the groups $\faktor{\Th{n}}{bP_{n+1}}$ and $\pi_{n}(G/O)$ are the same and are as given in \cite{TodaBook} and \cite{Ravenel}.
 \end{itemize}

\subsection{Organization}
In Section \ref{sec2} and \ref{sec 2.5}, we give the homotopy splitting of the $10^\text{th}$ Postnikov section of $pl/o$ and compute the set $[M^n,PL/O]$ for $8\leq n\leq 10$. In Section \ref{sec 3} we discuss the concordance and homotopy inertia group of smooth manifold $ M^n $, in particular, compute $I_h(\rp{n})$, for $ 8\leq n\leq10 $.


\section{Smooth structures on $8,9$-manifolds}\label{sec2}
In this section, we use the structure of the Postnikov section $\btau_{\leq 9} PL/O$ given in \cite{jahrennote}, and compute $[M, PL/O]$ for $\dim(M)=8,9$. For an 8-dimensional manifold, we deduce the following theorem. This is also implied by the computations in \cite{jahrennote}; however, here we independently confirm this result through a direct calculation.
\begin{thrm}\label{thm:8-manifold decomposition}
Let $ M^8 $ be a closed smooth manifold. Then $$[M^8,PL/O]\cong H^7(M^8;\sfrac{\Z}{28})\opl{}H^8(M^8;\sfrac{\Z}{2})$$
\end{thrm}
\begin{proof}
We prove that the $8^{\mathrm{th}}$-Postnikov section of $PL/O$ splits as a product $K(\sfrac{\Z}{28},7)\times K(\sfrac{\Z}{2}, 8)$, implying the required isomorphism.  It suffices to prove splitting in a $p$-local category for every prime $p$. As the $k$-invariants lie in the stable range, we work stably using Eilenberg-MacLane spectra instead of their underlying spaces. Note that, the homotopy groups of $pl/o$ in degrees at most $ 8$ have $p$-torsion only for $p=2$ and $7$. For the case $p=7$, the homotopy group is non-zero only in degree $7$. Therefore, it suffices to work $ 2 $-locally. 

The stable $ 8^\mathrm{th} $-Postnikov section of $ {pl/o} $ is the fiber of a map $ \Sigma^7 H\sfrac{\Z}{4} \lr \Sigma^9 H\sfrac{\Z}{2}$. Up to homotopy, this map is either  $0$ or $ (\Sigma^7 Sq^2)\circ q$, that is
\[
	\Sigma^7 H\sfrac{\Z}{4} \xrightarrow{\,\,q\,\,}\Sigma^7 H\sfrac{\Z}{2} \xrightarrow{\Sigma^7 Sq^2} \Sigma^9 H\sfrac{\Z}{2}. 
\]

This can be readily seen from the following diagram, wherein for any $\phi \in \{ \Sigma^7 H\sfrac{\Z}{4}, \Sigma^9 H\sfrac{\Z}{2} \}$, the observation $\{H\sfrac{\Z}{2}, \Sigma^2 H\sfrac{\Z}{2}\} \cong \{0,Sq^2\}$ gives us 
$$ \xymatrix@R-=0.5cm{\Sigma^{6} H\sfrac{\Z}{2} \ar[r]^-{\beta} & \Sigma^7 H\sfrac{\Z}{2} \ar[dr]_{\psi} \ar[r] & \Sigma^7 H\sfrac{\Z}{4} \ar[d]^{\phi} \ar[r]^{q} & \Sigma^7 H\sfrac{\Z}{2}\ar@{-->}[dl]^{\zeta} \\
&  &  \Sigma^9 H\sfrac{\Z}{2}   &}$$
where $\psi=\Sigma^7 Sq^2$ or $0$, and $\beta = Sq^1$. Since $  Sq^2\circ \beta \neq 0$, the only possibility for the map $\psi$ is $0$, which implies the existence of the $\zeta$ map. We intend to show that $(\Sigma^7 Sq^2)\circ q$ does not occur as the $k$-invariant. 
The idea is to exhibit a map $ \Sigma^7 H\sfrac{\Z}{4} \lr \btau_{\leq8}~pl/o $ such that its composition with the fibration map $b:\btau_{\leq8}~pl/o\lr \Sigma^7 H\sfrac{\Z}{4} $ is an equivalence. This implies that the $k$-invariant is $0$. 

Consider the natural fibration map $ \Sigma^{-1}g/pl\lr pl/o $ and restrict it to the $ 6 $-connected cover $ ^{> 6}\btau\Sigma^{-1}g/pl $ of $ \Sigma^{-1}g/pl $. Since $ \bpi_{7}(^{> 6}\btau\Sigma^{-1}g/pl)\cong\Z $ and $ \bpi_{8}(^{> 6}\btau\Sigma^{-1}g/pl)=0 $, implies
$${}^{>6}\btau_{\leq8}(\btau \Sigma^{-1}g/pl)\cong \Sigma^7 H\Z \stackrel{g}{\lr} \btau_{\leq8}~pl/o.$$ 
Consequently, we have the following diagram 
$$\xymatrix@R-=1.2em{
    \Sigma^7 H\Z  \ar[r]^-{\times 4} & \Sigma^7 H\Z \ar[d]^-{g} \ar[r]^-{p} & \Sigma^7 H\sfrac{\Z}{4} \ar[d]^-{}  \\
    \Sigma^8 H\sfrac{\Z}{2} \ar[r]_-{} & \btau_{\leq8}~pl/o \ar[r]_-{b} & \Sigma^7 H\sfrac{\Z}{4} 
}$$
where $p$ is the natural projection from $\Z$ onto $\Z/4$. Observe that the composition $b\circ g\circ (\times 4)$ is $0$, since the composition $p \circ (\times 4)$ is $0$.
Therefore we get a homotopy commutative diagram
$$\xymatrix@R-=1.2em{
   \Sigma^7 H\Z \ar@{-->}[d] \ar[r]^-{\times 4} & \Sigma^7 H\Z \ar[d]_-{g} \ar[r]^-{p} & \Sigma^7 H\sfrac{\Z}{4} \ar[d]^-{\cong} \\
   \Sigma^8 H\sfrac{\Z}{2} \ar[r]_-{} & \btau_{\leq8}~pl/o \ar[r]_-{b} & \Sigma^7 H\sfrac{\Z}{4} 
}$$

Since every map from $\Sigma^7 H\Z\lr \Sigma^8 H\sfrac{\Z}{2}$ is null homotopic, we have the following diagram
$$\xymatrix@R-=1.2em{
  \Sigma^7 H\Z \ar[dr]_-{0} \ar[r]^-{\times4} & \Sigma^7 H\Z \ar[d]_-{g} \ar[r]^-{p} & \Sigma^7 H\sfrac{\Z}{4} \ar@{-->}[dl]^-{\wi{b}} \\
    & \btau_{\leq 8}~pl/o  & 
}$$
where $g\circ(\times4)=0$ implies the existence of a map $\wi{b}$ having the desired property. 

Thus, the map $\wi{b}$ is a homotopy section for the map $b $, implying the following decomposition
\begin{equation}\label{eq: 8 pl/o decomp}
    \btau_{\leq8}~pl/o\simeq  \Sigma^7 H\sfrac{\Z}{4}\vee  \Sigma^8 H\sfrac{\Z}{2}.
\end{equation}
This complete the proof.\end{proof}


Theorem \ref{thm:8-manifold decomposition} shows how a decomposition result for the $8^{th}$ Postnikov section facilitates the computation of $\mathcal{C}(M)=[M,PL/O]$ for $8$-manifolds $M$. We now recall from \cite{jahrennote} the formula for ${}^{>6}\btau_{\leq 9}~ pl/o$. Note that, $\Ext(\sfrac{\Z}{4},\sfrac{\Z}{2})\cong \sfrac{\Z}{2}$ along with the universal coefficient theorem gives $H^{n+1}(K(\sfrac{\Z}{4},n);\sfrac{\Z}{2})\cong \sfrac{\Z}{2}$. Let us fix the notation $d_2$ for the corresponding stable map $ H\sfrac{\Z}{4}\lr \Sigma H\sfrac{\Z}{2}$, and define
\begin{equation}\label{notF}
\begin{aligned}
\FF & \coloneqq \mbox{Fibre}(H\sfrac{\Z}{2} \xrightarrow{Sq^2}\Sigma^2 H\sfrac{\Z}{2}) \\
\FF_2 & \coloneqq  \mbox{Fibre}(H\sfrac{\Z}{4} \xrightarrow{Sq^2\circ d_2} \Sigma^3 H\sfrac{\Z}{2}).
\end{aligned}
\end{equation}
With this, the $9^{th}$ Postnikov section of $pl/o$ is given by 
\begin{equation}\label{9th}
\btau_{\leq9}~pl/o \simeq \Sigma^8\FF \vee \Sigma^7\FF_2 \vee \Sigma^7 H\sfrac{\Z}{7} \vee \Sigma^9 H\sfrac{\Z}{2}.
\end{equation}


Let $ M^9 $ be a closed oriented smooth manifold. Then the minimal cell structure \cite[\S 4.C]{hatcher} on $ \sfrac{M^9}{M^{(6)}} $ is of the following form 
$$\sfrac{M^9}{M^{(6)}}\simeq\Big(\vee_{(p,r)\in J}M(\sfrac{\Z}{p^r},7)\vee(\Sp{7})^{\vee_l}\vee(\Sp{8})^{\vee_k}\Big)\bigcup e^9,$$
where $ M(\sfrac{\Z}{p^r},n) $ stands for the Moore space for the group $ \sfrac{\Z}{p^r} $ in degree $ n $, and $ J $ is some finite indexing set. In this case, the attaching map of the $ 9 $-cell onto $ 8 $-cells is null homotopic. Therefore,  
\begin{equation}\label{eq:M9/M6 "f" homotopy decomposition}
    \sfrac{M^9}{M^{(6)}}\simeq \Big((\vee_{(p,r)\in J}M(\sfrac{\Z}{p^r},7)\vee(\Sp{7})^{\vee_l})\bigcup_f e^9\Big)\vee(\Sp{8})^{\vee_k}~,
\end{equation}
and hence the attaching map $f$ lies in $ \bpi_{8}\Big(\vee_{(p,r)\in J}M(\sfrac{\Z}{p^r},7)\vee(\Sp{7})^{\vee_l}\Big) $.
By the connectivity argument,
$$ \bpi_{8}\Big(\vee_{(p,r)\in J}M(\sfrac{\Z}{p^r},7)\vee(\Sp{7})^{\vee_l}\Big) \cong\underset{(p,r)\in J}{\oplus}\bpi_{8}(M(\sfrac{\Z}{p^r},7))\underset{l}{\oplus}\bpi_{8}(\Sp{7})$$ with $\bpi_{8}(\Sp{7})\cong{\Z}/{2}\{\eta\}$ and
 \begin{center}
     $\bpi_8 (M(\Z/p^r,7)) \cong \begin{cases} 0 & \text{ if }  $p$ \text{ is odd, }\\ 
 \sfrac{\Z}{2}\{\iota\circ \eta\} & \text{ if }  $p=2$, \end{cases}$ 
 \end{center}  
 where $\iota\circ \eta$ is the composite $S^8 \stackrel{\eta}{\lr} S^7 \stackrel{\iota}\lr M(\sfrac{\Z}{2}^r,7)$.

Consider the Steenrod square operation
$Sq^2:H^7(M^9;\sfrac{\Z}{2})\lr H^{9}(M^9;\sfrac{\Z}{2})$, and for each $r\geq 1$, there is a higher order Bockstein operation $\beta_r:H^*(M^9;\sfrac{\Z}{2})\lr H^{*+1}(M^9;\sfrac{\Z}{2})$. Now, depending on the attaching map we have following possibilities (for more details  see \cite{li2023suspension}):
\begin{enumerate}
    \item If $ M^9 $ is a spin manifold then the attaching map $ \Sp{8}\lr\vee_{(p,r)\in J}M(\sfrac{\Z}{p^r},7)\vee(\Sp{7})^{\vee_l} $ is null homotopic, since $ Sq^2:H^7(M^9;\sfrac{\Z}{2})\lr H^9(M^9;\sfrac{\Z}{2}) $ is zero. Thus 
    \begin{equation}\label{eq:M9/M6 spin-case homotopy decomposition}
    \sfrac{M^9}{M^{(6)}}\simeq (\Sp{7})^{\vee_l}\vee(\Sp{8})^{\vee_k}\vee_{(p,r)\in J}M(\sfrac{\Z}{p^r},7)\vee \Sp{9}~,
\end{equation}
\item If $ M^9 $ is a non-spin manifold then $ Sq^2:H^7(M^9;\sfrac{\Z}{2})\lr H^9(M^9;\sfrac{\Z}{2}) $ is non-zero. In addition, 
\begin{enumerate}
    \item Suppose that for any $u\in H^7(M^9;\sfrac{\Z}{2})$ with $Sq^2(u)\neq0$ and any $v\in \kr{Sq^2}$, we have 
$\beta_r(u+v)=0$ and $u+v\notin \im{\beta_s},~~\forall r,s\geq 1$. Then, the non-trivial factor of the attaching map $f$ in \eqref{eq:M9/M6 "f" homotopy decomposition} is $\eta $, which implies
\begin{equation}\label{eq:M9/M6 "n" case homotopy decomposition}
    \sfrac{M^9}{M^{(6)}}\simeq C(\eta)\vee M',
\end{equation}
where $M'\simeq (\Sp{7})^{\vee_{l-1}}\vee(\Sp{8})^{\vee_k}\vee_{(p,r)\in J}M(\sfrac{\Z}{p^r},7)$.

\item Suppose that for any $u\in H^7(M^9;\sfrac{\Z}{2})$ with $Sq^2(u)\neq0$ and any $v\in \kr{Sq^2}$, we have
$u+v\notin \im{\beta_s},~~\forall s\geq 1$, while there exist $u'\in  H^7(M^9;\sfrac{\Z}{2})$ with $Sq^2(u')\neq 0$ and $v'\in \kr{Sq^2} $ such that $\beta_r(u'+v')\neq 0$ for some $r\geq 1 $. Then the only non-trivial factor of the attaching map $f$ in \eqref{eq:M9/M6 "f" homotopy decomposition} is $\iota\circ\eta$, and 
\begin{equation}\label{eq:M9/M6 "ion" case homotopy decomposition}
    \sfrac{M^9}{M^{(6)}}\simeq C(\iota\circ\eta)\vee M''
\end{equation}
where $M''\simeq (\Sp{7})^{\vee_l}\vee(\Sp{8})^{\vee_k}\vee_{(p,r)\in J'}M(\sfrac{\Z}{p^r},7)$.
\end{enumerate}
\end{enumerate}
 

The following theorem applies the splitting of the Postnikov section \eqref{9th} to the case of $9$-manifolds.
\begin{thrm}\label{9-mfld pl decomposition thm}
Let $ M^9 $ be a closed oriented smooth manifold and let $Sq^2\circ d_2 : H^6(M^9;\sfrac{\Z}{4}) \lr H^9(M^9;\sfrac{\Z}{2})$. \begin{enumerate}
\item If $ M^9 $ is a spin manifold then
$$[M^9,PL/O]\cong H^7(M^9;\sfrac{\Z}{28})\opl{}H^8(M^9;\sfrac{\Z}{2})\opl{}H^9(M^9;\sfrac{\Z}{2}\opl{}\sfrac{\Z}{2}\opl{}\sfrac{\Z}{2}). $$

\item Let $ M^9 $ be a non-spin manifold. 
\begin{enumerate}
    \item If $Sq^2\circ d_2 $ is non-trivial then
    $$[M^9,PL/O]\cong H^7(M^9;\sfrac{\Z}{28})\opl{}H^8(M^9;\sfrac{\Z}{2})\opl{}H^9(M^9;\sfrac{\Z}{2}).$$

    \item If $ Sq^2\circ d_2 $ is trivial then 
    $$[M^9,PL/O]\cong H^7(M^9;\sfrac{\Z}{7})\opl{}H^8(M^9;\sfrac{\Z}{2})\opl{}H^9(M^9;\sfrac{\Z}{2})\opl{}  [M^9,\Sigma^7\FF_2] ~,$$ where
    $[M,\Sigma^7\FF_2]\cong \wi{K}\oplus \wi{A}$, with $\wi{K}$ is a part of the short exact sequence
   \begin{center}
       \begin{tikzcd}
	0 & {\kr{Sq^2}} & {\wi{K}} & {\kr{Sq^2\oplus Sq^1}} & 0,
	\arrow[from=1-1, to=1-2]
	\arrow[from=1-2, to=1-3]
	\arrow[from=1-3, to=1-4]
	\arrow[from=1-4, to=1-5]
\end{tikzcd}
   \end{center}
    $Sq^1:H^7(M;\sfrac{\Z}{2})\lr H^8(M;\sfrac{\Z}{2})$ and  $Sq^2:H^7(M;\sfrac{\Z}{2})\lr H^9(M;\sfrac{\Z}{2})$, and $\wi{A}$ is the non-trivial extension satisfying the short exact sequence 
\begin{center}
    \begin{tikzcd}
	0 & {\sfrac{\Z}{2}} & {\wi{A}} & A & 0,
	\arrow[from=1-1, to=1-2]
	\arrow[from=1-2, to=1-3]
	\arrow[from=1-3, to=1-4]
	\arrow[from=1-4, to=1-5]
\end{tikzcd}
\end{center}
where $A=\sfrac{\Z}{4}$ or $\sfrac{\Z}{2}$. 

\end{enumerate}

\end{enumerate}
\end{thrm}

\begin{proof}
Using \eqref{9th}, we have
$[M^9,PL/O]\cong[M^9,\btau_{\leq9}~pl/o]\cong H^7(M^9;\sfrac{\Z}{7})\opl{}\\H^9(M^9;\sfrac{\Z}{2})\opl{}[M^9,\Sigma^8 \FF]\oplus [M^9,\Sigma^7\FF_2].$
Thus, it is enough to compute $ [M^9,\Sigma^7 \FF_2] $ and $[M^9,\Sigma^8\FF]$.

Note that, $\Sigma^8\FF$ is $7$-connected, thus  
\begin{equation}\label{eq: [M9,F]=[M9/M6,F]}
    [M^9,\Sigma^8\FF]\cong[\sfrac{M^9}{M^{(6)}},\Sigma^8\FF].
\end{equation}
The space $\Sigma^7 \FF_2$ is $6$-connected and the group $[M^9,\Sigma^7 \FF_2]$ can be computed using the following commutative diagram:
\begin{equation}\label{dig:[M,F2] to [M/M5,F2] comm dig}
  \begin{tikzcd}[column sep=small]
	{[M^9,\Sigma^6H\sfrac{\Z}{4}]} & {[M^9,\Sigma^9H\sfrac{\Z}{2}]} & {[M^9,\Sigma^7\FF_2]} & {[M^9,\Sigma^7H\sfrac{\Z}{4}]} \\
	0 & {[\sfrac{M^9}{M^{(5)}},\Sigma^9H\sfrac{\Z}{2}]} & {[\sfrac{M^9}{M^{(5)}},\Sigma^7\FF_2]} & {[\sfrac{M^9}{M^{(5)}},\Sigma^7H\sfrac{\Z}{4}]}
	\arrow["{Sq^2\circ d_2}", from=1-1, to=1-2]
	\arrow[from=1-2, to=1-3]
	\arrow[two heads, from=1-3, to=1-4]
	\arrow[from=2-1, to=1-1]
	\arrow[from=2-1, to=2-2]
	\arrow["\cong"', from=2-2, to=1-2]
	\arrow[from=2-2, to=2-3]
	\arrow[from=2-3, to=1-3]
	\arrow["\cong"', from=2-4, to=1-4]
	\arrow[two heads, from=2-3, to=2-4]
\end{tikzcd}
\end{equation}

\begin{enumerate}
    \item Let $ M^9 $ be a spin manifold. Then \eqref{eq:M9/M6 spin-case homotopy decomposition} and \eqref{eq: [M9,F]=[M9/M6,F]} together gives
$$[{M^9},\Sigma^8\FF]\cong[\sfrac{M^9}{M^{(6)}},\Sigma^8\FF]\cong H^8(\sfrac{M^9}{M^{(6)}};\sfrac{\Z}{2})\opl{}H^9(\sfrac{M^9}{M^{(6)}};\sfrac{\Z}{2}).$$
For $[M^9,\Sigma^7 \FF_2]$, since $Sq^2=0$, using \eqref{dig:[M,F2] to [M/M5,F2] comm dig} we get the following commutative diagram of short exact sequences
\begin{equation}\label{dig:short exact seq [M,F2] to [M/M5,F2] comm dig}
  \begin{tikzcd}[column sep=small]
	0 & {[M^9,\Sigma^9H\sfrac{\Z}{2}]} & {[M^9,\Sigma^7\FF_2]} & {[M^9,\Sigma^7H\sfrac{\Z}{4}]} & 0 \\
	0 & {[\sfrac{M^9}{M^{(5)}},\Sigma^9H\sfrac{\Z}{2}]} & {[\sfrac{M^9}{M^{(5)}},\Sigma^7\FF_2]} & {[\sfrac{M^9}{M^{(5)}},\Sigma^7H\sfrac{\Z}{4}]} & 0~.
	\arrow[from=1-2, to=1-3]
	\arrow[from=1-3, to=1-4]
	\arrow[from=2-1, to=2-2]
	\arrow["\cong"', from=2-2, to=1-2]
	\arrow[from=2-2, to=2-3]
	\arrow[from=2-3, to=1-3]
	\arrow["\cong"', from=2-4, to=1-4]
	\arrow[from=2-3, to=2-4]
	\arrow[from=1-1, to=1-2]
	\arrow[from=1-4, to=1-5]
	\arrow[from=2-4, to=2-5]
\end{tikzcd}
\end{equation}
This implies 
$$[M^9,\Sigma^7\FF_2]\cong[\sfrac{M^9}{M^{(5)}},\Sigma^7\FF_2] .$$ 
In \eqref{dig:short exact seq [M,F2] to [M/M5,F2] comm dig}, we demonstrate the splitting of the short exact sequence in the second row, thereby implying the splitting of the sequence in the first row.
For that purpose, consider the cofiber sequence $ \sfrac{M^9}{M^{(5)}}\lr \sfrac{M^9}{M^{(6)}}\xrightarrow{\Psi} \Sigma(\sfrac{M^{(6)}}{M^{(5)}})$, and the following commutative diagram induced from it
\begin{equation}\label{comm dig: M9/M5 M/M6 SM6/M5}
 \begin{tikzcd}[column sep=0.5em]
	0 & {[\sfrac{M^9}{M^{(5)}},\Sigma^9H\sfrac{\Z}{2} ]} & {[\sfrac{M^9}{M^{(5)}},\Sigma^7\FF_2 ]} & {[\sfrac{M^9}{M^{(5)}},\Sigma^7H\sfrac{\Z}{4} ]} & 0 \\
	0 & {[\sfrac{M^9}{M^{(6)}}, \Sigma^9H\sfrac{\Z}{2}]} & {[\sfrac{M^9}{M^{(6)}},\Sigma^7\FF_2 ]} & {[\sfrac{M^9}{M^{(6)}},\Sigma^7H\sfrac{\Z}{4} ]} & 0 \\
	0 & {[\Sigma(\sfrac{M^{(6)}}{M^{(5)}}), \Sigma^9H\sfrac{\Z}{2}]} & {[\Sigma(\sfrac{M^{(6)}}{M^{(5)}}),\Sigma^7\FF_2 ]} & {[\Sigma(\sfrac{M^{(6)}}{M^{(5)}}),\Sigma^7H\sfrac{\Z}{4} ]} & 0.
	\arrow[from=1-1, to=1-2]
	\arrow[from=1-2, to=1-3]
	\arrow[from=1-3, to=1-4]
	\arrow[from=1-4, to=1-5]
	\arrow[from=2-1, to=2-2]
	\arrow[from=2-2, to=2-3]
	\arrow[from=2-3, to=2-4]
	\arrow[from=2-4, to=2-5]
	\arrow[from=2-3, to=1-3]
	\arrow["\cong"', from=2-2, to=1-2]
	\arrow[two heads, from=2-4, to=1-4]
	\arrow["{\phi^*}"', from=3-2, to=2-2]
	\arrow["{\Psi^*}"', from=3-3, to=2-3]
	\arrow["{\gamma^*}"', from=3-4, to=2-4]
	\arrow[from=3-3, to=3-4]
	\arrow[from=3-2, to=3-3]
	\arrow[from=3-1, to=3-2]
	\arrow[from=3-4, to=3-5]
\end{tikzcd}
\end{equation}
Observe that due to \eqref{eq:M9/M6 spin-case homotopy decomposition}, the map $\Psi:\sfrac{M^9}{M^{(6)}}\lr \Sigma(\sfrac{M^{(6)}}{M^{(5)}})\simeq \vee_i\Sp{7}$ decomposes in $\gamma:\vee_{l}\Sp{7}\vee_k\Sp{8}\vee_{(p,r)\in J}M(\sfrac{\Z}{p^r},7)\lr \vee_i\Sp{7}$ and $\phi: \Sp{9}\lr \vee_i\Sp{7}$. Up to homotopy, the map $\phi$ is either $0$ or $\eta^2$ ($\pi_2^s\cong \sfrac{\Z}{2}\lara{\eta^2}$). If it is $\eta^2=\eta\circ\eta$, then the map $\phi^*$ is 0 due to the following.
\begin{center}
    \begin{tikzcd}[sep=small]
	{[\Sp{7},\Sigma^7\FF_{2}]} && {[\Sp{9},\Sigma^7\FF_{2}]} \\
	& {[\Sp{8},\Sigma^7\FF_{2}]=0}
	\arrow["{\eta^*}"', from=1-1, to=2-2]
	\arrow["{\eta^*}"', from=2-2, to=1-3]
	\arrow["{(\eta^2)^*}", from=1-1, to=1-3]
\end{tikzcd}    
\end{center}
Thus, we obtain $\im{\Psi^*}=\im{\gamma^*}$. This shows that in \eqref{comm dig: M9/M5 M/M6 SM6/M5}, the exact sequence in the second row induces a split exact sequence at $\faktor{[\sfrac{M^9}{M^{(6)}},\Sigma^7\FF_2]}{\im{\Psi^*}}$. Therefore, we get 
\begin{equation}\label{eq: M9 spin case [M9,F2]=H7+H9}
    [M^9,\Sigma^7\FF_2]\cong H^7(M;\sfrac{\Z}{4})\oplus H^9(M;\sfrac{\Z}{2}).
\end{equation}
    As a result, we obtain the following in the spin case,
$$[M^9,PL/O]\cong H^7(M^9;\sfrac{\Z}{28})\opl{}H^8(M^9;\sfrac{\Z}{2})\opl{}H^9(M^9;\sfrac{\Z}{2}\opl{}\sfrac{\Z}{2}\opl{}\sfrac{\Z}{2}).$$

\item Let $ M^9 $ be a non-spin manifold.
Using \eqref{eq: [M9,F]=[M9/M6,F]}, we have 
 $[M^9,\Sigma^8\FF]\cong[\sfrac{M^9}{M^{(6)}},\Sigma^8\FF]$, and by \eqref{eq:M9/M6 "n" case homotopy decomposition} and \eqref{eq:M9/M6 "ion" case homotopy decomposition},
\begin{equation}\label{eq: [M9/M6,F] =[cone,F]+[M',F]or..}
    [\sfrac{M^9}{M^{(6)}},\Sigma^8\FF]\cong [C(\eta),\Sigma^8\FF]\oplus[M',\Sigma^8\FF]\text{ or }[C(\iota\circ\eta),\Sigma^8\FF]\oplus[M'',\Sigma^8\FF].
\end{equation} Further, $[\Sp{8},\Sigma^8 \FF]
\overset{\cong}{\lr}[M(\sfrac{\Z}{2^r},7),\Sigma^8 \FF]\cong\sfrac{\Z}{2}$, and $[M(\sfrac{\Z}{p^r},7),\Sigma^8 \FF]=0$ for all odd prime $p$. Hence 
\begin{equation}\label{eq: 9-dim [M',FF]=H8 }
    [M',\Sigma^8\FF]\cong H^8(M';\sfrac{\Z}{2})\text{ and } [M'',\Sigma^8\FF]\cong H^8(M'';\sfrac{\Z}{2}).
\end{equation}
It remains to compute $[C(\eta),\Sigma^8\FF]$ and $[C(\iota\circ\eta),\Sigma^8\FF]$. 

The computation of $[C(\eta),\Sigma^8\FF]$ follows easily from the following exact sequence 
\begin{center}
    \begin{tikzcd}
	 {\cdots~[\Sp{8},\Sigma^8\FF]} & {[\Sp{9},\Sigma^8\FF]} & {[C(\eta),\Sigma^8\FF]} & {[\Sp{7},\Sigma^8\FF]=0.}
	\arrow["\cong", from=1-1, to=1-2]
	\arrow[from=1-2, to=1-3]
	\arrow[from=1-3, to=1-4]
\end{tikzcd}
\end{center}
Therefore, we get
\begin{equation}\label{eq:[cone,F]=0 9-dim}
    [C(\eta),\Sigma^8\FF]=0
\end{equation}

For the computation of $[C(\iota\circ\eta),\Sigma^8\FF]$, consider the following commutative diagram obtained from the cofiber sequence $\Sp{8}\overset{\iota\circ \eta}{\lr}M(\sfrac{\Z}{2}^r,7)\lr C(\iota\circ\eta)$,
\begin{equation}
    \label{seq: [c',F] exact seq in 9-dim}
 \begin{tikzcd}[column sep=1.3em]
    & {[\Sp{8},\Sigma^8\FF]} \\
	{[\Sigma M(\sfrac{\Z}{2}^r,7),\Sigma^8\FF]} & {[\Sp{9},\Sigma^8\FF]} & {[C(\iota\circ\eta),\Sigma^8\FF]} & {[ M(\sfrac{\Z}{2}^r,7),\Sigma^8\FF]} & {[\Sp{8},\Sigma^8\FF]} \\
	&&& {[\Sp{8},\Sigma^8\FF]}
	\arrow["\Sigma i^*", two heads, from=2-1, to=1-2]
	\arrow[from=2-1, to=2-2]
	\arrow["\cong", from=1-2, to=2-2]
	\arrow[from=2-2, to=2-3]
	\arrow[from=2-3, to=2-4]
	\arrow["d^*", "\cong"', from=3-4, to=2-4]
	\arrow["(\iota\circ\eta)^*", from=2-4, to=2-5]
	\arrow["\times 2"', from=3-4, to=2-5]
\end{tikzcd}
\end{equation}
where the map $\Sigma i^*$ is a part of the exact sequence obtained from the cofiber sequence $\Sp{7}\overset{2^r}{\lr}\Sp{7}\overset{i}{\lr}M(\sfrac{\Z}{2}^r,7)$.
Also, the map $(\iota\circ\eta)^*$ is trivial which implies that \\$[C(\iota\circ\eta),\Sigma^8\FF]\cong[ M(\sfrac{\Z}{2}^r,7),\Sigma^8\FF]\cong\sfrac{\Z}{2}$. 
 In conclusion, we get 
\begin{equation}\label{eq: [Cn,F]=H8(Cn), [Cin,F]=H8(Cin) }
     [C(\eta),\Sigma^8 \FF]\cong H^8(C(\eta);\sfrac{\Z}{2}) \text{ and } [C(\iota\circ\eta),\Sigma^8 \FF]\cong H^8(C(\iota\circ\eta);\sfrac{\Z}{2}).
 \end{equation}
Therefore, combining \eqref{eq: 9-dim [M',FF]=H8 } and \eqref{eq: [Cn,F]=H8(Cn), [Cin,F]=H8(Cin) } we obtain 
\begin{equation}\label{eq: [M,F]=H8(M)}
    [M,\Sigma^8\FF]\cong H^8(M;\sfrac{\Z}{2}).
\end{equation}

Finally, let us compute $[M^9, \Sigma^7\FF_2]$, by taking into account two cases depending on the nature of the attaching map: whether the map $Sq^2\circ d_2: H^6(M;\sfrac{\Z}{4})\lr H^9(M;\sfrac{\Z}{2})$ is non-trivial or trivial. 

In the case when $Sq^2\circ d_2$ is non-trivial, it is clear from the diagram \eqref{dig:[M,F2] to [M/M5,F2] comm dig} that
\begin{equation}\label{eq: 9case non-spin with sqd2=0 eqaulity in F2 case}
    [M^9, \Sigma^7\FF_2]\cong[M^9, \Sigma^7H\sfrac{\Z}{4}].
\end{equation}

Now, for the case $Sq^2\circ d_2=0$, we need $[C(\eta),\Sigma^7\FF_2]$ and $[C(\iota\circ\eta),\Sigma^7\FF_2]$. So let us first calculate these groups. 

We have $\eta:\Sp{8}\lr \Sp{7}$, $\FF_2=\text{Fiber}(H\sfrac{\Z}{4}\overset{Sq^2\circ d_2}{\lr}\Sigma^3H\sfrac{\Z}{2})$, and the following commutative square
\begin{equation}\label{dig: sq2sq1=sq2d2}
 \begin{tikzcd}[column sep=1.5cm]
	{H\sfrac{\Z}{2}} & {\Sigma^3H\sfrac{\Z}{2}} \\
	{H\sfrac{\Z}{4}} & {\Sigma^3H\sfrac{\Z}{2}}
	\arrow[from=1-1, to=2-1]
	\arrow["{=}", from=1-2, to=2-2]
	\arrow["{Sq^2\circ Sq^1}", from=1-1, to=1-2]
	\arrow["{Sq^2\circ d_2}", from=2-1, to=2-2]
\end{tikzcd}
\end{equation}
Let $\mathcal{D}=\text{Fiber}(H\sfrac{\Z}{2}\overset{Sq^2\circ Sq^1}{\lr}\Sigma^3H\sfrac{\Z}{2})$. 
We have the following commutative diagram
\begin{center}
   \begin{tikzcd}[column sep=1.2cm]
	{\Sigma^7 H\sfrac{\Z}{4}} & {\Sigma^7 \mathcal{D}} & {\Sigma ^8\mathcal{F}} \\
	{\Sigma^7 H\sfrac{\Z}{4}} & {\Sigma^7 H\sfrac{\Z}{2}} & {\Sigma^8 H\sfrac{\Z}{2}} \\
	0 & {\Sigma^{10}H\sfrac{\Z}{2}} & {\Sigma^{10}H\sfrac{\Z}{2}}
	\arrow["{=}", from=3-2, to=3-3]
	\arrow["{\Sigma^7(Sq^2\circ Sq^1)}", from=2-2, to=3-2]
	\arrow[from=1-3, to=2-3]
	\arrow["{\Sigma^7Sq^2}", from=2-3, to=3-3]
	\arrow["{\Sigma^7Sq^1}", from=2-2, to=2-3]
	\arrow[from=1-2, to=2-2]
	\arrow[from=1-1, to=1-2]
	\arrow["{=}", from=1-1, to=2-1]
	\arrow[from=1-2, to=1-3]
	\arrow[from=2-1, to=3-1]
	\arrow[from=3-1, to=3-2]
	\arrow[from=2-1, to=2-2]
\end{tikzcd}
\end{center}
in which the rows and columns are cofiber sequences.

Note that 
$[C(\eta),\Sigma^7 \mathcal{F}]= 0 = [C(\eta),\Sigma^8\mathcal{F}]$ implies
$$[C(\eta),\Sigma^7 \mathcal{D}]\cong [C(\eta),\Sigma^7 H\sfrac{\Z}{4}]\cong\sfrac{\Z}{4}.$$

Now, using \eqref{dig: sq2sq1=sq2d2} we have the following commutative diagram of short exact sequences, wherein observe that if the bottom row splits, then so does the top row.
\begin{center}
   \begin{tikzcd}[row sep=2em]
	0 & {[C(\eta),\Sigma^9H\sfrac{\Z}{2}]} & {[C(\eta),\Sigma^7 \mathcal{D}]} & {[C(\eta),\Sigma^7 H\sfrac{\Z}{2}]} & 0 \\
	0 & {[C(\eta),\Sigma^9H\sfrac{\Z}{2}]} & {[C(\eta),\Sigma^7 \FF_2]} & {[C(\eta),\Sigma^7 H\sfrac{\Z}{4}]} & 0
	\arrow[from=1-1, to=1-2]
	\arrow[from=2-1, to=2-2]
	\arrow["{=}", from=1-2, to=2-2]
	\arrow[from=1-2, to=1-3]
	\arrow[from=1-3, to=2-3]
	\arrow[from=2-2, to=2-3]
	\arrow[from=1-3, to=1-4]
	\arrow[hook, from=1-4, to=2-4]
	\arrow[from=2-3, to=2-4]
	\arrow[from=2-4, to=2-5]
	\arrow[from=1-4, to=1-5]
\end{tikzcd}
\end{center}
Therefore, 
\begin{equation}\label{eq:non-spin case [c,f2] when sqd not 0}
    [C(\eta),\Sigma^7 \FF_2]\cong\sfrac{\Z}{8}.
\end{equation}

To compute $[C(\iota\circ\eta), \Sigma^7 \FF_2]$, consider the following commutative diagram
\begin{equation}\label{dig: comparing [Cn,F2] and [Cin,F2]}
 \begin{tikzcd}[row sep=2em]
	0 & {{\overset{\sfrac{\Z}{2}}{\overset{\rotatebox{90}{$ \cong$}}{[\Sp{9},\Sigma^7\FF_2]}}}} & {[C(\iota\circ\eta),\Sigma^7\FF_2]} & {[M(\sfrac{\Z}{2}^r,7),\Sigma^7\FF_2]} & 0 \\
	0 & {[\Sp{9},\Sigma^7\FF_2]} & {[C(\eta),\Sigma^7\FF_2]} & {[\Sp{7},\Sigma^7\FF_2]} & 0
	\arrow["=", from=1-2, to=2-2]
	\arrow[from=1-2, to=1-3]
	\arrow[from=1-3, to=1-4]
	\arrow[from=1-4, to=1-5]
	\arrow[from=1-3, to=2-3]
	\arrow[tail, from=1-4, to=2-4]
	\arrow[from=2-2, to=2-3]
	\arrow[from=2-3, to=2-4]
	\arrow[from=2-4, to=2-5]
	\arrow[from=1-1, to=1-2]
	\arrow[from=2-1, to=2-2]
\end{tikzcd}
\end{equation}
where the rows are exact sequences induced from the cofiber sequences of $C(\iota\circ\eta)$ and $C(\eta)$. 
To compute $[M(\sfrac{\Z}{2}^r,7),\Sigma^7\FF_2]$, we  use the fiber sequence $\eqref{notF}$ of $\FF_2$ that gives $$[M(\sfrac{\Z}{2}^r,7),\Sigma^7\FF_2]\cong H^7(M(\sfrac{\Z}{2}^r,7);\sfrac{\Z}{4}),$$ and further, compute the cohomology using cofiber sequence $\Sp{7} \overset{\times2^r}{\lr}\Sp{7}\lr M(\sfrac{\Z}{2}^r,7)$,  that gives
\begin{equation}
    [M(\sfrac{\Z}{2}^r,7),\Sigma^7\FF_2]\cong \begin{cases}
        \sfrac{\Z}{4} &  \text{ if } r=1,\\
        \sfrac{\Z}{2} &  \text{ if } r>1.
    \end{cases}
\end{equation}

A straightforward diagram chasing in \eqref{dig: comparing [Cn,F2] and [Cin,F2]} along with \eqref{eq:non-spin case [c,f2] when sqd not 0} shows the non-splitting of the short exact sequence at $[C(\iota\circ\eta),\Sigma^7\FF_2]$. Furthermore, 
\begin{equation}\label{eq:non-spin case [c',f2] when sqd not 0}
    [C(\iota\circ\eta^2),\Sigma^7\FF_2]\cong \begin{cases}
        \sfrac{\Z}{8} &  \text{ if } r=1,\\
        \sfrac{\Z}{4} &  \text{ if } r>1.
    \end{cases}
\end{equation}

Now, consider the case when $Sq^2\circ d_2$ is trivial.
This implies $Sq^2Sq^1:H^6(M;\sfrac{\Z}{2})\to H^9(M;\sfrac{\Z}{2})$ is 0.
Therefore $$Sq^2\colon\ckr{Sq^1\colon H^6(M;\sfrac{\Z}{2})\to H^7(M;\sfrac{\Z}{2})}\lr H^9(M;\sfrac{\Z}{2})$$ is well-defined and non-zero. Therefore, we have
$$\ckr{Sq^1}=\kr{Sq^2}\opl{}\sfrac{\Z}{2}.$$

Let $q:H^*(M;\sfrac{\Z}{4})\lr H^*(M;\sfrac{\Z}{2})$ be the map induced by the non-trivial morphism $\sfrac{\Z}{4}\lr\sfrac{\Z}{2}$.

Note that $\im{H^7(\widehat{M};\sfrac{\Z}{4})\to H^7(M;\sfrac{\Z}{4})} \subseteq \kr{Sq^2\circ q}$, where $\widehat{M}=M'$ or $M''$ as mentioned in \eqref{eq:M9/M6 "n" case homotopy decomposition} and \eqref{eq:M9/M6 "ion" case homotopy decomposition}.
Then, using the exact sequence 
\begin{center}
    \begin{tikzcd}[sep=small]
	0 & {\ckr{Sq^1}} & {H^7(M;\sfrac{\Z}{4})} & {\kr{Sq^1:H^7(M;\sfrac{\Z}{2})\to H^8(M;\sfrac{\Z}{2})}} & 0,
	\arrow[from=1-1, to=1-2]
	\arrow[from=1-2, to=1-3]
	\arrow["q", from=1-3, to=1-4]
	\arrow[from=1-4, to=1-5]
\end{tikzcd}
\end{center}
observe that $H^7(M;\sfrac{\Z}{4}) =\wi{K}\opl{} A$ where $A= \sfrac{\Z}{2} \text{ or } \sfrac{\Z}{4}$, and $\wi{K}$ fits into the following possible non-trivial extension which is determined from the structure of $H^7(M;\sfrac{\Z}{4})$
\begin{center}
    \begin{tikzcd}
	0 & {\kr{Sq^2}} & {\wi{K}} & {\kr{Sq^2\oplus Sq^1}} & 0.
	\arrow[from=1-1, to=1-2]
	\arrow[from=1-2, to=1-3]
	\arrow[from=1-3, to=1-4]
	\arrow[from=1-4, to=1-5]
\end{tikzcd}
\end{center}
Additionally, note that $\wi{K}\subseteq H^7(M;\sfrac{\Z}{4})$, is in fact the image of $H^7(\widehat{M};\sfrac{\Z}{4})$. Thus, we have \begin{equation}
    [M,\Sigma^7\FF_2]\cong \wi{K}\oplus \wi{A}
\end{equation}
where $\wi{A}$ is the non-trivial extension in  
\begin{center}
    \begin{tikzcd}
	0 & {\sfrac{\Z}{2}} & {\wi{A}} & A & 0.
	\arrow[from=1-1, to=1-2]
	\arrow[from=1-2, to=1-3]
	\arrow[from=1-3, to=1-4]
	\arrow[from=1-4, to=1-5]
\end{tikzcd}
\end{center}
\end{enumerate}
This completes the proof.
 \end{proof}


\section{Smooth structures on $10$-manifolds }\label{sec 2.5}

We now consider at the analogue of Theorem \ref{9-mfld pl decomposition thm} for simply-connected $10$-manifolds. To address this, we need information about the $10^{th}$-Postnikov section $\tau_{\leq 10} ~pl/o$ of $pl/o$ computed in \cite{jahrennote}. In this context, observe that $Sq^2\circ Sq^2 \circ d_2=0$ leads to the construction of a class $\Phi : \FF_2 \lr \Sigma^4 H\sfrac{\Z}{2}$ using the following diagram 
\[\xymatrix@R-=0.6cm{ \Sigma^{-1} H\sfrac{\Z}{4} \ar[r]^{Sq^2\circ d_2} & \Sigma^2 H\sfrac{\Z}{2} \ar[d]^{Sq^2} \ar[r] & \mbox{$\FF_2$} \ar@{-->}[ld]^{\Phi} \\ 
                                      & \Sigma^4 H\sfrac{\Z}{2} } \]
The operation $\Phi$ defines a secondary cohomology operation from
$ \kr{Sq^2\circ d_2}(\subseteq H^i(M;\sfrac{\Z}{4})) \lr H^{i+4}(M;\sfrac{\Z}{2})$. Now, let
\begin{equation}\label{notE}
\EE \coloneqq \mbox{Fibre}(\FF_2 \stackrel{\Phi}{\lr}\Sigma^4 H\sfrac{\Z}{2}).
\end{equation}
With this, the $10^{th}$ Postnikov section of $pl/o$ is given by 
\begin{equation}\label{10th}
\btau_{\leq 10}~pl/o \simeq \Sigma^8\FF \vee \Sigma^7\EE \vee \Sigma^7 H\sfrac{\Z}{7} \vee \Sigma^9 H\sfrac{\Z}{2}\vee \Sigma^{10}H\Z/3.
\end{equation}

Let $ M^{10} $ be a closed smooth manifold with $ H_1(M^{10}) =0$. Then there is a minimal cell structure \cite[\S 4.C]{hatcher} on $ \sfrac{M^{10}}{M^{(6)}} $ of the form 
$$\sfrac{M^{10}}{M^{(6)}}\simeq \Big({\vee_l}\Sp{7}\vee_k\Sp{8}\vee_{(p,r)\in J}M(\sfrac{\Z}{p^r},7)\Big)\bigcup_f e^{10},$$
where the attaching map $f$ lies in $ \bpi_{9}\Big({\vee_l}\Sp{7}\vee_k\Sp{8}\vee_{(p,r)\in J}M(\sfrac{\Z}{p^r},7)\Big) $.
By the connectivity argument,
\begin{equation}\label{eq: 10-dim "f" attaching map homo decom}
    \bpi_{9}\Big({\vee_l}\Sp{7}\vee_k\Sp{8}\vee_{(p,r)\in J}M(\sfrac{\Z}{p^r},7)\Big) \cong{\oplus_l}\bpi_{9}(\Sp{7}) {\oplus_k}\bpi_{9}(\Sp{8}){\oplus_{(p,r)\in J}}\bpi_{9}(M(\sfrac{\Z}{p^r},7)),
\end{equation} 
with $\bpi_{9}(\Sp{8})\cong{\Z}/{2}\{\eta\}$, $\bpi_{9}(\Sp{7})\cong{\Z}/{2}\{\eta^2\}$ and $\bpi_{9}(M(\Z/p^r,7))\cong{\Z}/{2}\{\iota\circ\eta^2\}$.

If $M^{10}$ is a spin manifold, then there exists a  higher order cohomology operation $\psi:H^6(M;\sfrac{\Z}{4})\lr H^{10}(M;\sfrac{\Z}{2})$  corresponding to $(Sq^2\circ Sq^2)+(Sq^3\circ Sq^1)=0$ in order to detect the map $\eta^2$ \cite[Corollary 2, pg177]{mosher2008cohomology}. Depending on either $\psi$ is trivial or not, we have following possibilities:
\begin{enumerate} 
    \item If $\psi$ is trivial, then
    \begin{equation}\label{eq: 10-dim-spin "f" null homotopic case}
        \sfrac{M^{10}}{M^{(6)}}\simeq \sfrac{M^{10}}{M^{(6)}}\simeq {\vee_l}\Sp{7}\vee_k\Sp{8}\vee_{(p,r)\in J}M(\sfrac{\Z}{p^r},7)\vee\Sp{10}.
    \end{equation}

    \item  If $\psi$ is non-trivial, then 
    \begin{equation}\label{eq: 10-dim-spin "f" n2 case}
        \sfrac{M^{10}}{M^{(6)}}\simeq C(\eta^2)\vee M'  ,
    \end{equation}  or  
    \begin{equation}\label{eq: 10-dim-spin "f" ion2 case}
        \sfrac{M^{10}}{M^{(6)}}\simeq C(\iota\circ\eta^2)\vee M'',
    \end{equation}
    with $M'\simeq {\vee_{l-1}}\Sp{7}\vee_k\Sp{8}\vee_{(p,r)\in J}M(\sfrac{\Z}{p^r},7)$ and $M''\simeq {\vee_{l}}\Sp{7}\vee_k\Sp{8}\vee_{(p,r)\in J'}M(\sfrac{\Z}{p^r},7)$. 
   
\end{enumerate}

The following theorem applies the splitting of the Postnikov section \eqref{10th} to the case of $10$-manifolds.

\begin{thrm}\label{pl decomposition of 10 mfld}
Let $ M^{10} $ be a closed  smooth $ 10 $-manifold with $H_1(M)=0$, and $\Phi,\psi: H^6(M;\sfrac{\Z}{4})\lr H^{10}(M;\sfrac{\Z}{2})$ be the secondary operations described in \eqref{notE} and \eqref{eq: 10-dim-spin "f" null homotopic case}.
\begin{enumerate}
    \item Let $ M^{10} $ be a spin manifold. 
    
    \begin{enumerate}
        \item If $\Phi=0$ then
    \begin{equation*}
        \begin{split}
            [M^{10},PL/O] \cong & H^7(M^{10};\sfrac{\Z}{7})\opl{}H^8(M^{10};\sfrac{\Z}{2})\opl{}H^9(M^{10};\sfrac{\Z}{2})\opl{}H^{10}(M^{10};\sfrac{\Z}{3}) \\ & \oplus [M^{10},\Sigma^7\EE].
        \end{split}
    \end{equation*}   
Furthermore, if the higher order cohomology operation $\psi=0$ then $$[M^{10},\Sigma^7\EE]\cong H^{10}(M;\sfrac{\Z}{2})\oplus H^7(M;\sfrac{\Z}{4}).$$
On the other hand, if $\psi\neq 0$, then $[M,\Sigma^7\EE]=\wi{K}\oplus \wi{A}$, where $\wi{K}\subseteq H^7(M;\sfrac{\Z}{4})$, and $\wi{A}$ is the non-trivial extension satisfying the following sequence 
\begin{center}
    \begin{tikzcd}
	0 & {\sfrac{\Z}{2}} & {\wi{A}} & A & 0.
	\arrow[from=1-1, to=1-2]
	\arrow[from=1-2, to=1-3]
	\arrow[from=1-3, to=1-4]
	\arrow[from=1-4, to=1-5]
\end{tikzcd}
\end{center}
with $A=\sfrac{\Z}{4}$ or $\sfrac{\Z}{2}$.

\item If $\Phi\neq0$ then $$[M^{10},PL/O]\cong H^7(M^{10};\sfrac{\Z}{28})\opl{}H^8(M^{10};\sfrac{\Z}{2})\opl{}H^9(M^{10};\sfrac{\Z}{2})\opl{}H^{10}(M^{10};\sfrac{\Z}{3}).$$
\end{enumerate}

\item If $ M^{10} $ is a non-spin manifold then 
$$[M^{10},PL/O]\cong H^7(M^{10};\sfrac{\Z}{7})\opl{}H^{10}(M^{10};\sfrac{\Z}{3}) \opl{}\kr{Sq^2}\opl{} \kr{Sq^2\circ d_2},$$
where $Sq^2:H^8(M^{10};\sfrac{\Z}{2})\to H^{10}(M^{10};\sfrac{\Z}{2})$, and $Sq^2\circ d_2:H^7(M;\sfrac{\Z}{4})\to H^{10}(M;\sfrac{\Z}{2})$.
\end{enumerate}
   
\end{thrm}
\begin{proof}
Using the decomposition in \eqref{10th}, it is enough to compute $[M^{10},\Sigma^8\FF]$ and $[M^{10},\Sigma^7\EE]$.

For $[M^{10},\Sigma^8\FF]$, we have a long exact sequence from fibration \eqref{notF}
\[ \xymatrix{\cdots \ar[r] & {\underset{0}{\underset{\rotatebox{90}{$ \cong$}}{[M,\Sigma^9 H\sfrac{\Z}{2}]}}}  \ar[r] & [M,\mbox{$\Sigma^8 \FF$}] \ar[r] & [M,\Sigma^8 H\sfrac{\Z}{2}]\ar[r]^{Sq^2} & [M,\Sigma^{10} H\sfrac{\Z}{2}] } \] 
which implies
\begin{equation}
    [M^{10},\Sigma^8\FF]= \begin{cases} H^8(M;\sfrac{\Z}{2}) & \text{ if }  M \text{ is spin,}\\ \kr{Sq^2} & \text{ if } M \text{ is not spin.}\end{cases}
\end{equation}
    For $[M,\Sigma^7\EE]$, we have the following exact sequence obtained from \eqref{notE}
\begin{equation}\label{seq: exact sequence for EE}
    \xymatrix@R-=0.2cm@C-=0.6cm{[ M,\Sigma^8H\sfrac{\Z}{2}] \ar[d] \ar[rd]^{Sq^2} & & & & \\
 [M,\Sigma^6 \mbox{$\FF_2$}]\ar[r]^-{\Phi} & {\underset{\sfrac{\Z}{2}}{\underset{\rotatebox{90}{$ \cong$}}{[M,\Sigma^{10} H\sfrac{\Z}{2}]}}} \ar[r] & [M,\mbox{$\Sigma^7 \EE$}] \ar[r] & [M,\Sigma^7 \mbox{$\FF_2$}]\ar[r]^-{\Phi} & {\underset{0}{\underset{\rotatebox{90}{$ \cong$}}{[M,\Sigma^{11} H\sfrac{\Z}{2}]}}}}
\end{equation}
If either $M$ is non-spin or the map $\Phi\neq 0$, it is evident that $[M,\Sigma^7\EE]\cong [M,\Sigma^7\FF_2]$. The latter group can be computed using the long exact sequence 
\begin{equation}\label{seq: seq of F in non-spin E 10case}  
    \xymatrix{\cdots~ {\underset{0}{\underset{\rotatebox{90}{$ \cong$}}{[M,\Sigma^9 H\sfrac{\Z}{2}]}}}  \ar[r] & [M,\mbox{$\Sigma^7 \FF_2$}] \ar[r] & [M,\Sigma^7 H\sfrac{\Z}{4}]\ar[r]^{Sq^2\circ d_2} & [M,\Sigma^{10} H\sfrac{\Z}{2}], }
\end{equation}
which together with \eqref{seq: exact sequence for EE} gives  
\begin{equation}
  [M,\Sigma^7\EE] \cong [M,\Sigma^7\FF_2]=\kr{Sq^2 \circ d_2}.
\end{equation}
Thus \begin{equation}
    [M,\Sigma^7 \EE]= \begin{cases}  H^7(M;\sfrac{\Z}{4}) & \text{ if } M \text{ is spin and } \Phi\neq 0\\  
    \kr{Sq^2 \circ d_2} &  \text{ if } M \text{ is non-spin . }\end{cases}
\end{equation}

Now, assume that $M$ is spin and $\Phi=0$. In this case, from \eqref{seq: exact sequence for EE} and \eqref{seq: seq of F in non-spin E 10case}, we obtain the following short exact sequence
\begin{equation}\label{seq: exact seq 10-dim case [M,S7E]} 
    0 \lr H^{10}(M;\sfrac{\Z}{2}) \lr [M,\Sigma^7 \EE] \lr H^7(M;\sfrac{\Z}{4}) \lr 0.
\end{equation}

In order to compute $[M,\Sigma^7 \EE]$, we first compute $[C(\eta^2),\Sigma^7\EE]$ and $[C(\iota\circ\eta^2),\Sigma^7\EE]$. In this regard, consider the following commutative diagram whose rows and columns are cofiber sequences
\begin{center}
    \begin{tikzcd}
	{\Sp{9}} & {\Sp{7}} & {C(\eta^2)} \\
	{\Sp{8}} & {\Sp{7}} & {C(\eta)} \\
	{\Sigma C (\eta)} & {*} & {\Sigma^2 C(\eta)}
	\arrow[Rightarrow, no head, from=1-2, to=2-2]
	\arrow["\eta"', from=1-1, to=2-1]
	\arrow["{\eta^2}", from=1-1, to=1-2]
	\arrow["\eta"', from=2-1, to=2-2]
	\arrow[from=1-2, to=1-3]
	\arrow[from=2-2, to=2-3]
	\arrow[from=1-3, to=2-3]
	\arrow[from=2-2, to=3-2]
	\arrow[from=2-1, to=3-1]
	\arrow[from=3-1, to=3-2]
	\arrow[from=3-2, to=3-3]
	\arrow[from=2-3, to=3-3]
\end{tikzcd}
\end{center}

This, together with \eqref{seq: exact seq 10-dim case [M,S7E]} gives the following
\begin{center}
    \begin{tikzcd}
	&& {[C(\eta),\Sigma^7\EE]} \\
	0 & {\sfrac{\Z}{2}} & {[C(\eta^2),\Sigma^7\EE]} & {\sfrac{\Z}{4}} & 0 \\
	&& {[\Sigma C(\eta),\Sigma^7\EE]}
	\arrow[from=1-3, to=2-3]
	\arrow[from=2-3, to=3-3]
	\arrow[from=2-1, to=2-2]
	\arrow[from=2-2, to=2-3]
	\arrow[from=2-3, to=2-4]
	\arrow[from=2-4, to=2-5]
\end{tikzcd}
\end{center}
For the non-spin case, we have $[C(\eta),\Sigma^7\EE]\cong [C(\eta),\Sigma^7\FF_2]\cong\sfrac{\Z}{8} $, and $[\Sigma C(\eta),\Sigma^7\EE]=0$ using \eqref{eq:non-spin case [c,f2] when sqd not 0}. This implies that the short exact sequence for $C(\eta^2)$ does not split, and we get \begin{equation}
    [C(\eta^2),\Sigma^7\EE]\cong\sfrac{\Z}{8}.
\end{equation}
To compute $[C(\iota\circ \eta^2),\Sigma^7\EE]$, consider the following commutative diagram
\begin{equation}\label{dig: 10dim cone eta^2 and ioeta^2 commutative}
  \begin{tikzcd}[row sep=2em]
	0 & {\overset{\sfrac{\Z}{2}}{\overset{\rotatebox{90}{$ \cong$}}{[\Sp{10},\Sigma^7\EE]}}} & {[C(\iota\circ \eta^2),\Sigma^7\EE]} & {[M(\Z/2^r,7),\Sigma^7\EE]} & 0 \\
	0 & {[\Sp{10},\Sigma^7\EE]} & {[C(\eta^2),\Sigma^7\EE]} & {\underset{\sfrac{\Z}{4}}{\underset{\rotatebox{90}{$ \cong$}}{[\Sp{7},\Sigma^7\EE]}}} & 0
	\arrow[from=2-1, to=2-2]
	\arrow[from=2-2, to=2-3]
	\arrow[from=2-3, to=2-4]
	\arrow[from=2-4, to=2-5]
	\arrow[from=1-1, to=1-2]
	\arrow[from=1-2, to=1-3]
	\arrow[from=1-3, to=1-4]
	\arrow[from=1-4, to=1-5]
	\arrow["=", from=1-2, to=2-2]
	\arrow[from=1-3, to=2-3]
	\arrow[tail, from=1-4, to=2-4]
\end{tikzcd}
\end{equation}
which is obtained from the cofiber sequences of $C(\eta^2)$ and $C(\iota\circ\eta^2)$. Since the short exact sequence at $[C(\eta^2),\Sigma^7\EE]$ does not split, we conclude that the short exact sequence at $[C(\iota\circ \eta^2),\Sigma^7\EE]$ does not split as well.

Thus, depending on $r$ in $M(\Z/2^r,7)$, \eqref{dig: 10dim cone eta^2 and ioeta^2 commutative} implies
\begin{equation}
    [C(\iota\circ\eta^2),\Sigma^7\EE]\cong \begin{cases}
        \sfrac{\Z}{8} &  \text{ if } r=1,\\
        \sfrac{\Z}{4} &  \text{ if } r>1.
    \end{cases}
\end{equation}

Note that the $10$-manifold spin case bifurcates into two sub-cases: whether the map $\psi: H^6(M;\sfrac{\Z}{4})\to H^{10}(M;\sfrac{\Z}{2})$ (as described in \eqref{eq: 10-dim-spin "f" null homotopic case}) is non-trivial or trivial.

Assume that $\psi\neq 0$ for spin $10$-manifold. This implies that the attaching map of the top cell $\Sp{9}\lr \sfrac{M}{M^{(6)}}\simeq (\Sp{7})^{\vee_l}\vee(\Sp{8})^{\vee_s}\vee_{p,r}M(\sfrac{\Z}{p^r},7)$ attaches by $\eta^2$ onto some $(\Sp{7})^{\vee_l}$ or some $M(\sfrac{\Z}{p^r},7)$.

Let $\wi{K}\subseteq H^7(M;\sfrac{\Z}{4})$ be $\kr{H^7(M;\sfrac{\Z}{4})\to H^7(M;\sfrac{\Z}{2})\overset{\psi}{\to} H^{10}(M;\sfrac{\Z}{2})}$.
Further, let $H^7(M;\sfrac{\Z}{4})=\wi{K}\oplus A$ where $A=\sfrac{\Z}{2}$ or $\sfrac{\Z}{4}$ (similar to the $9$-manifold scenario).
Then, $[M,\Sigma^7\EE]=\wi{K}\oplus \wi{A}$ where $\wi{A}$ is the non-trivial extension in $0\lr \sfrac{\Z}{2}\lr \wi{A}\lr A\lr 0$.


On the other hand, assume $\psi=0$.
 Analogous to the $9$-manifold spin case, we get the following commutative diagram of short exact sequences
\begin{equation}\label{dig:short exact seq [M10,E] to [M/M5,E] comm dig}
\begin{tikzcd}[column sep=small]
	0 & {[M^{10},\Sigma^{10}H\sfrac{\Z}{2}]} & {[M^{10},\Sigma^7\EE]} & {[M^{10},\Sigma^7H\sfrac{\Z}{4}]} & 0 \\
	0 & {[\sfrac{M^{10}}{M^{(5)}},\Sigma^{10}H\sfrac{\Z}{2}]} & {[\sfrac{M^{10}}{M^{(5)}},\Sigma^7\EE]} & {[\sfrac{M^{10}}{M^{(5)}},\Sigma^7H\sfrac{\Z}{4}]} & 0~.
	\arrow[from=1-2, to=1-3]
	\arrow[from=1-3, to=1-4]
	\arrow[from=2-1, to=2-2]
	\arrow["\cong"', from=2-2, to=1-2]
	\arrow[from=2-2, to=2-3]
	\arrow[from=2-3, to=1-3]
	\arrow["\cong"', from=2-4, to=1-4]
	\arrow[from=2-3, to=2-4]
	\arrow[from=1-1, to=1-2]
	\arrow[from=1-4, to=1-5]
	\arrow[from=2-4, to=2-5]
\end{tikzcd}
\end{equation}
This implies that
$$[M^{10},\Sigma^7\EE]\cong[\sfrac{M^{10}}{M^{(5)}},\Sigma^7\EE] .$$ 
For the computation of $[\sfrac{M^{10}}{M^{(5)}},\Sigma^7\EE ]$, consider the following 
\begin{equation}\label{comm dig: M10/M5 M10/M6 SM6/M5}
  \begin{tikzcd}[column sep=small]
	0 & {[\sfrac{M^{10}}{M^{(5)}},\Sigma^{{10}}H\sfrac{\Z}{2} ]} & {[\sfrac{M^{10}}{M^{(5)}},\Sigma^7\EE ]} & {[\sfrac{M^{10}}{M^{(5)}},\Sigma^7H\sfrac{\Z}{4} ]} & 0 \\
	0 & {[\sfrac{M^{10}}{M^{(6)}}, \Sigma^{{10}}H\sfrac{\Z}{2}]} & {[\sfrac{M^{10}}{M^{(6)}},\Sigma^7\EE ]} & {[\sfrac{M^{10}}{M^{(6)}},\Sigma^7H\sfrac{\Z}{4} ]} & 0 \\
	0 & {[\Sigma(\sfrac{M^{(6)}}{M^{(5)}}), \Sigma^{10}H\sfrac{\Z}{2}]} & {[\Sigma(\sfrac{M^{(6)}}{M^{(5)}}),\Sigma^7\EE ]} & {[\Sigma(\sfrac{M^{(6)}}{M^{(5)}}),\Sigma^7H\sfrac{\Z}{4} ]} & 0.
	\arrow[from=1-1, to=1-2]
	\arrow[from=1-2, to=1-3]
	\arrow[from=1-3, to=1-4]
	\arrow[from=1-4, to=1-5]
	\arrow[from=2-1, to=2-2]
	\arrow[from=2-2, to=2-3]
	\arrow[from=2-3, to=2-4]
	\arrow[from=2-4, to=2-5]
	\arrow[from=2-3, to=1-3]
	\arrow["\cong"', from=2-2, to=1-2]
	\arrow[two heads, from=2-4, to=1-4]
	\arrow["{\phi^*}"', from=3-2, to=2-2]
	\arrow["{\Psi^*}"', from=3-3, to=2-3]
	\arrow["{\gamma^*}"', from=3-4, to=2-4]
	\arrow[from=3-3, to=3-4]
	\arrow[from=3-2, to=3-3]
	\arrow[from=3-1, to=3-2]
	\arrow[from=3-4, to=3-5]
\end{tikzcd}
\end{equation}


Using \eqref{eq: 10-dim-spin "f" null homotopic case}, the map $\Psi:\sfrac{M^{10}}{M^{(6)}}\lr \Sigma(\sfrac{M^{(6)}}{M^{(5)}})\simeq \vee_i\Sp{7}$ decomposes in $\gamma:\vee_{l}\Sp{7}\vee_k\Sp{8}\vee_{(p,r)\in J}M(\sfrac{\Z}{p^r},7)\lr \vee_i\Sp{7}$ and $\phi: \Sp{10}\lr \vee_i\Sp{7}$.
Moreover, note that the attaching map of the $10$-cell onto the $6$-cell is a multiple of $\nu\in \bpi_3^s$. Thus, we need to compute $\nu^*:[\Sp{7},\Sigma^7\EE]\lr [\Sp{10},\Sigma^7\EE] $.

Using \eqref{10th}, we get the following commutative diagram (observe that it suffices to work 2-locally)
\begin{center}
    \begin{tikzcd}
	{[\Sp{7},\Sigma^7\EE]} & {[\Sp{10},\Sigma^7\EE]} \\
	{[\Sp{7},\btau_{\leq 10}~pl/o]} & {[\Sp{10},\btau_{\leq 10}~pl/o]} \\
	{[\Sp{7},{}^{>6}\btau_{\leq 10}~g/o]} & {[\Sp{10},{}^{>6}\btau_{\leq 10}~g/o]}
	\arrow["{\nu^*}", from=1-1, to=1-2]
	\arrow["{\nu^*}", from=3-1, to=3-2]
	\arrow[Rightarrow, no head, from=1-1, to=2-1]
	\arrow["{\nu^*}", from=2-1, to=2-2]
	\arrow[Rightarrow, no head, from=1-2, to=2-2]
	\arrow[from=2-1, to=3-1]
	\arrow[from=2-2, to=3-2]
\end{tikzcd}
\end{center}
Since the map $[\Sp{10},\btau_{\leq 10}~pl/o]\lr [\Sp{10},{}^{>6}\btau_{\leq 10}~g/o]$ is an isomorphism, and $[\Sp{7},{}^{>6}\btau_{\leq 10}~g/o]=0$, the map $\nu^*=0$, and hence in \eqref{comm dig: M10/M5 M10/M6 SM6/M5} the induced map $\phi^*=0$. Thus, $\im{\Psi^*}=\im{\gamma^*}$, which implies in the diagram \eqref{comm dig: M10/M5 M10/M6 SM6/M5} the sequence splits at $\sfrac{[\sfrac{M^{10}}{M^{(6)}},\Sigma^7\EE]}{\im{\Psi^*}}$ and consequently at $[\sfrac{M^{10}}{M^{(5)}},\Sigma^7\EE]$. Therefore the exact sequence in \eqref{dig:short exact seq [M10,E] to [M/M5,E] comm dig} splits for $M$.

This completes the proof for all $10$ dimensional manifold $M^{10}$ with $H_1(M^{10})=0$.\end{proof}


\section{Inertia group of $M^n$ for $8\leq n\leq 10$}\label{sec 3}

In the classification of smooth structures of manifolds, the group $\Theta_n$ plays an important role. Specifically, if we take a connected sum of a smooth manifold with an exotic sphere, the resulting manifold remains homeomorphic to the underlying topological manifold; however, it might change the diffeomorphism class. Therefore, the determination of subgroups of $\Theta_n$ namely, the inertia group, the homotopy inertia group, and the concordance inertia group correlates with the classification problem. We begin with the definitions of these groups.


\begin{defn} Let $M^n$ be a closed  oriented smooth manifold.
\begin{enumerate}[(i)]
    \item  The \emph{inertia group} of $M^n$ is the subgroup $I(M^n )\subseteq \Th{n} $ of homotopy spheres $\Sigma^n$ such that $M^n\csum\Sigma^n$ is diffeomorphic to $M^n$.

    \item The \emph{homotopy inertia group} $I_h(M^n)$ consists of $\Sigma^n\in I(M^n)$ such that there exists a diffeomorphism from $M^n\csum\Sigma^n\to M^n$ that is homotopic to the canonical homeomorphism $h_{\Sigma^n}:M^n\csum\Sigma^n\to M^n$.

    \item The \emph{concordance inertia group} $I_c(M^n)$ consists of $\Sigma^n\in I_h(M^n)$ such that there exists a diffeomorphism from $M^n\csum\Sigma^n\to M^n$ that is concordant to the canonical homeomorphism $h_{\Sigma^n}:M^n\csum\Sigma^n\to M^n$.
\end{enumerate}
\end{defn}
 This section discusses the concordance and homotopy inertia groups of smooth manifold $ M^n $ for $ 8\leq n\leq10 $.

\subsection{Concordance inertia group}

Recall that, if $d: M^n\lr \Sp{n}$ is a degree one map then the kernel of the induced map $d^*:[\Sp{n},PL/O]\lr [M^n,PL/O]$ can be identified with $I_c(M^n)$.
\begin{thrm}\label{thm: Ic8}
    Let $M^8$ be a closed manifold. Then $$I_c(M^8)=\{0\}.$$ 
\end{thrm}
\begin{proof}
     The statement follows directly from the $8^\text{th}$ Postnikov decomposition \eqref{eq: 8 pl/o decomp} of $pl/o$, wherein the map $d^*:[\Sp{8},\btau_{\leq8}~pl/o]\lr [M^8,\btau_{\leq{8}}~pl/o] $ induced by the collapse map $d:M^8\lr M^8 \setminus\mathit{int}(\D{8})$ is injective.
\end{proof}

\begin{thrm}\label{thm: Ic(9)}
    Let $M^9$ be a closed oriented smooth manifold with $Sq^2\circ d_2:H^6(M;\sfrac{\Z}{4})\lr H^9(M;\sfrac{\Z}{2})$. 
    \begin{enumerate}
        \item If $M^9$ is spin then $$I_c(M^9)=\{0\}.$$ 

        \item 
       If $M^9$ is non-spin and $Sq^2\circ d_2=0$ then $$I_c(M^9)=\sfrac{\Z}{2}\lara{\eta\circ\epsilon}.$$

              \item If $M^9$ is non-spin and $Sq^2\circ d_2\neq0$ then $$I_c(M^9)=\sfrac{\Z}{2}\lara{\eta\circ\epsilon}\oplus\sfrac{\Z}{2}\lara{\mu}.$$
        \end{enumerate}
\end{thrm}
\begin{proof} Using the splitting \eqref{9th}, we have the following decomposition 
\begin{equation}\label{eq: [M9,PL/O] decomposition}
    [M^9,PL/O]\cong H^7(M^9;\sfrac{\Z}{7})\opl{}H^9(M^9;\sfrac{\Z}{2})\opl{}[M^9,\Sigma^8 \FF]\oplus [M^9,\Sigma^7\FF_2].
\end{equation}
Since, the concordance inertia group of a manifold is isomorphic to the kernel of $d^*:[\Sp{9}, PL/O] \lr [M^9, PL/O]$, it is enough to check the $\kr{d^*}$ in \eqref{eq: [M9,PL/O] decomposition} componentwise.
As $M^9$ is oriented, the map $d^*:H^9(\Sp{9};\sfrac{\Z}{2})\lr H^9(M^9;\sfrac{\Z}{2})$ is injective on the top cohomology. Thus, we focus on the maps $d^*:[\Sp{9},\Sigma^8\FF]\lr [M^9,\Sigma^8\FF] $ and $d^*:[\Sp{9},\Sigma^7\FF_2]\lr [M^9,\Sigma^7\FF_2]$.
\begin{enumerate}[$(1)$]
    \item If $M^9$ is a spin manifold, then by the proof of \Cref{9-mfld pl decomposition thm}$(1)$, we have $[M^9,\Sigma^8\FF]\cong H^8(M^9;\sfrac{\Z}{2})\opl{}H^9(M^9;\sfrac{\Z}{2})$
and $[M^9,\Sigma^7\FF_2] \cong H^7(M;\sfrac{\Z}{4})\oplus H^9(M;\sfrac{\Z}{2})$. Thus, the induced maps $d^*:[\Sp{9},\Sigma^8\FF]\to [M^9,\Sigma^8\FF] $ and $d^*:[\Sp{9},\Sigma^7\FF_2]\lr [M^9,\Sigma^7\FF_2]$ are injective, with their image isomorphic to $H^9(M^9;\sfrac{\Z}{2})$.
Therefore, the kernel of $d^*:[\Sp{9},PL/O] \lr [M^9,PL/O]$ is trivial for the spin manifold $M^9$.

\item Let $M^9$ be a non-spin manifold. From \eqref{eq:M9/M6 "n" case homotopy decomposition} and \eqref{eq:M9/M6 "ion" case homotopy decomposition} 
we have $\sfrac{M^9}{M^{(6)}}\simeq C(\eta)\vee M'$ or $C(\iota\circ\eta)\vee M''$. Let $X= C(\eta)$ or $C(\iota\circ\eta)$. Since $M'$ or $M''$ is in the $8$-skeleton of $\sfrac{M^9}{M^{(6)}}$, the degree one map factors through $X$, and we have the following commutative diagram
\begin{center}
    \begin{tikzcd}[row sep=small]
	{[\Sp{9},\Sigma^8\FF]} && {[M^9,\Sigma^8\FF]} \\
	& {[X,\Sigma^8\FF]}
	\arrow["{d^*}", from=1-1, to=2-2]
	\arrow["{d^*}", from=1-1, to=1-3]
	\arrow[from=2-2, to=1-3]
\end{tikzcd}
\end{center}
It follows from \eqref{eq:[cone,F]=0 9-dim} and \eqref{seq: [c',F] exact seq in 9-dim} that the map $[\Sp{9},\Sigma^8\FF]\cong\sfrac{\Z}{2}\lara{\eta\circ\epsilon}\lr[X,\Sigma^8\FF] $ is trivial. Therefore, for the component $\Sigma^8\FF$, the induced map $d^*:[\Sp{9},\Sigma^8\FF]\lr[M^9,\Sigma^8\FF]$ is trivial as well.

For the map $d^*:[\Sp{9},\Sigma^7\FF_2]\lr [M^9,\Sigma^7\FF_2]$, consider the following commutative diagram 
\begin{equation}\label{dig: 9-dim Ic dig for F2}
  \begin{tikzcd}[row sep=2em]
	& {[\Sp{9},\Sigma^9H\sfrac{\Z}{2}]} & 
{[\Sp{9},\Sigma^7\FF_2]\cong\sfrac{\Z}{2}\lara{\mu}} \\
	{[M^9,\Sigma^6H\sfrac{\Z}{4}]} & {[M^{9},\Sigma^9H\sfrac{\Z}{2}]} & {[M^{9},\Sigma^7\FF_2]}
	\arrow["\cong", from=1-2, to=1-3]
	\arrow["d^*", "\cong"', from=1-2, to=2-2]
	\arrow["{d^*}", from=1-3, to=2-3]
	\arrow[from=2-2, to=2-3]
	\arrow["{Sq^2d_2}", from=2-1, to=2-2]
\end{tikzcd}
\end{equation}
whose rows are part of exact sequences obtained from the fibration sequence for $\FF_2$ as in \eqref{notF}. In the case $Sq^2\circ d_2=0$, the map $[M^{9},\Sigma^9H\sfrac{\Z}{2}]\lr[M^{9},\Sigma^7\FF_2]$ becomes injective, and so is the $d^*:[\Sp{9},\Sigma^7\FF_2]\lr [M^9,\Sigma^7\FF_2]$.

Further, the map $Sq^2\circ d_2$ in \eqref{dig: 9-dim Ic dig for F2} being non-trivial implies the map $[M^{9},\Sigma^9H\sfrac{\Z}{2}]\lr[M^{9},\Sigma^7\FF_2]$ is trivial.
As a consequence, the map $d^*:[\Sp{9},\Sigma^7\mathcal{F}_2]\lr [M^9,\Sigma^7\mathcal{F}_2]$ is trivial as well.

\end{enumerate} 

This completes the proof in all cases. \end{proof}

\begin{thrm}\label{thm: 10-dim Ic}
    Let $ M^{10} $ be a closed simply-connected smooth manifold, and let $\Phi: H^6(M^{10};\sfrac{\Z}{4})\to H^{10}(M^{10};\sfrac{\Z}{2})$ be the secondary cohomology operation mentioned in \eqref{notE}.    \begin{enumerate}
        \item If $ M^{10} $ is a spin manifold and $\Phi=0$, then $$I_c(M^{10})=\{0\}.$$

        \item If $ M^{10} $ is either a spin manifold with $\Phi\neq 0$, or a non-spin manifold, then $$I_c(M^{10})=\sfrac{\Z}{2}\lara{\eta\circ\mu}.$$

    \end{enumerate}
\end{thrm}
\begin{proof} The splitting \eqref{10th} gives the following 
\begin{equation*}
    \begin{split}
        [M^{10},PL/O]\cong\, & [M^{10},\Sigma^8\FF]\oplus [M^{10},\Sigma^7\EE] \oplus H^{7}(M^{10};\sfrac{\Z}{7})\oplus H^{9}(M^{10};\sfrac{\Z}{2}) \\ 
        & \oplus H^{10}(M^{10};\Z/3).
    \end{split}
\end{equation*}
By using \Cref{9-mfld pl decomposition thm}, we have $[\Sp{10},PL/O]\cong[\Sp{10},\Sigma^{10}H\sfrac{\Z}{3}]\oplus[\Sp{10},\Sigma^7\EE]$. Analogous to \Cref{thm: Ic(9)}, it is enough to check the componentwise kernel of $d^*:[\Sp{10},\Sigma^{10}H\sfrac{\Z}{3}]\oplus[\Sp{10},\Sigma^7\EE]\lr [M^{10},\Sigma^{10}H\sfrac{\Z}{3}]\oplus[M^{10},\Sigma^7\EE]$. 

Since $M^{10}$ is simply-connected, the map $d^*:[\Sp{10},\Sigma^{10}H\sfrac{\Z}{3}]\lr [M^{10},\Sigma^{10}H\sfrac{\Z}{3}]$ is injective. This implies that the generator $\beta_1$ of $[\Sp{10},\Sigma^{10}H\sfrac{\Z}{3}]=\sfrac{\Z}{3}\lara{\beta_1}$ does not belong to $I_c(M^{10})$.

So it remains to compute the kernel of $d^*:[\Sp{10},\Sigma^7\EE]=\sfrac{\Z}{2}\lara{\eta\circ\mu}\lr [M^{10},\Sigma^7\EE]$. For that, let us consider the following commutative diagram
\begin{equation}
 \begin{tikzcd}
	{[M^{10},\Sigma^8H\sfrac{\Z}{2}]} & {[\Sp{10},\Sigma^{10}H\sfrac{\Z}{2}]} & {[\Sp{10},\Sigma^{7}\EE]} \\
	{[M^{10},\Sigma^{(6)}\FF_2]} & {[M^{10},\Sigma^{10}H\sfrac{\Z}{2}]} & {[M^{10},\Sigma^{7}\EE]}
	\arrow["\cong", from=1-2, to=1-3]
	\arrow["\cong"', "{d^*}", from=1-2, to=2-2]
	\arrow[from=2-2, to=2-3]
	\arrow["{d^*}", from=1-3, to=2-3]
	\arrow["\Phi"', from=2-1, to=2-2]
	\arrow["{Sq^2}", from=1-1, to=2-2]
	\arrow[from=1-1, to=2-1]
\end{tikzcd}
\end{equation}
where the bottom row is a part of the exact sequence \eqref{seq: exact sequence for EE}. From the diagram, it is clear that the map $d^*:[\Sp{10},\Sigma^{7}\EE]\to [M^{10},\Sigma^{7}\EE]$ is injective or trivial if and only if the map from $[M^{10},\Sigma^{10}H\sfrac{\Z}{2}]\to [M^{10},\Sigma^{7}\EE] $ is injective or trivial, respectively.
For the spin case, whenever $\Phi=0$, the map $[M^{10},\Sigma^{10}H\sfrac{\Z}{2}]\lr [M^{10},\Sigma^{7}\EE] $ is injective, else it is trivial. 
On the other hand, $\Phi$ is always non-trivial for the non-spin case. Hence the map 
$[M^{10},\Sigma^{10}H\sfrac{\Z}{2}]\lr [M^{10},\Sigma^{7}\EE] $ is trivial.
This completes the proof in all cases. \end{proof}

\subsection{Homotopy inertia group}
 In surgery theory, there is a natural map $f_{M^n}:\Th{n}\lr  S^{\text{Diff}}(M^n)$, where $S^\text{Diff}(M^n)$ is the homotopy smooth structure set of $M^n$  \cite{Browder1972Surgery,Wall}. Then, the homotopy inertia group $I_h(M^n)$ can be identified with $\kr{f_{M^n}}$.
If $M^n$ is a closed oriented smooth manifold,then we have the following commutative square  
\begin{equation}\label{dig: commu sq for homo inertia grp of even dim mfld}
 \xymatrix@-=.7cm{
   [\Sp{n},{PL}/{O}] \ar[d]_-{f_{M^n}} \ar[r]^-{\psi_*} & [\Sp{n},{G}/{O}] \ar[d]^-{d^*}  \\
  S^{\text{Diff}}(M^n) \ar[r]_-{g'} & [M^n,{G}/{O}] ~,} \end{equation}
where $\psi:PL/O\lr G/O$ is a natural fibration, $d^*$ is induced from the degree one map, and $g'$ is the part of surgery exact sequence. Note that, if $n$ is even and $M^n$ is simply-connected then the maps $\psi_*$ and $g'$ are injective.

Recall that, $I_c(M^n)\subseteq I_h(M^n)$. Hence it is enough to discuss those elements in $\Th{n}\equiv[\Sp{n},{PL}/{O}]$ which are not in $I_c(M^n)$.

We observe the fact that follows from straightforward diagram chasing. Consider the following commutative diagram whose rows and columns are part of exact sequences of abelian groups, and $\xi$ and $\beta$ are surjective homomorphisms.
\begin{equation}
    \begin{tikzcd}
	A & B & C \\
	E & F & G \\
	& H
	\arrow["\alpha", from=1-1, to=1-2]
	\arrow["\beta", two heads, from=1-2, to=1-3]
	\arrow["\xi", two heads, from=1-1, to=2-1]
	\arrow["\gamma", from=2-1, to=2-2]
	\arrow["\eta", from=2-2, to=2-3]
	\arrow["\psi", from=1-3, to=2-3]
	\arrow["\phi", from=1-2, to=2-2]
	\arrow["\theta", from=2-2, to=3-2]
\end{tikzcd}
\end{equation}

\noindent\textbf{Claim}\label{claim algebraic}: If $y\not\in \im{\phi}$  and $\eta(y)=z\neq 0$ then $z\not\in \im{\psi}$.

 We prove this by contradiction. Suppose $z\in \im{\psi}$. Then there exist elements $b,c$ such that $\beta(b)=c$ and $\psi(c)=z$. For $\phi(b)=x$, we have $\eta(x)=z$ and $\theta(x)=0$. Hence $\eta(x-y)=0$, which also implies that there exists an element $e$ such that $\gamma(e)=x-y$. As $\xi$ is surjective, commutativity of diagram allows us to affirm the existence of $b'$ such that $\phi(b')=x-y$. Consequently, $\theta(x-y)=\theta(x)-\theta(y)=0$, implying that $\theta(y)=0$. However, this contradicts the fact that $y \not\in \text{im}(\phi)$. Therefore, the claim is established.

Let $M^{(6)}$ be the $6$-skeleton of an $n$-manifold $M$ for $n\leq 10$. Consider the following commutative diagram, 
\begin{center}
    \begin{tikzcd}
	{[\Sigma M^{(6)},{}^{>6}\btau_{\leq 10}(\Sigma^{-1}{g}/{pl})]} & {[\Sigma M^{(6)},K(\Z_{(2)},7)]} \\
	{[\Sigma M^{(6)},\btau_{\leq{10}}~pl/o]} & {[\Sigma M^{(6)},K(\sfrac{\Z}{4},7)]}
	\arrow["j_*", from=1-1, to=2-1]
	\arrow[Rightarrow, no head, from=1-1, to=1-2]
	\arrow[Rightarrow, no head, from=2-1, to=2-2]
	\arrow["j_*", two heads, from=1-2, to=2-2]
\end{tikzcd}
\end{center}
where the map $j_*$ is induced from the fibration 
\begin{equation}\label{fib: g/pl to pl/o to g/o}
    \xymatrix{ {}^{>6}\btau_{\leq 10}(\Sigma^{-1}{g}/{pl})\ar[r]^-{j} & \btau_{\leq 10}~{pl}/{o} \ar[r]^-{\psi} & {}^{>6}\btau_{\leq 10} ~{g}/{o}}.
\end{equation}
 Observe that, the map $j_*:[\Sigma M^{(6)},{}^{>6}\btau_{\leq 10}(\Sigma^{-1}{g}/{pl})]\lr [\Sigma M^{(6)},\btau_{\leq{10}}~pl/o]$ is surjective.

Further, consider the following commutative diagram
    \begin{center}\label{dig: appling claim to M}
     \begin{tikzcd}[column sep=1em, row sep=2em]
	{[\Sigma M^{(6)}, {}^{>6}\btau_{\leq 10}(\Sigma^{-1}{g}/{pl})]} & {[\sfrac{M}{M^{(6)}}, {}^{>6}\btau_{\leq10}(\Sigma^{-1}{g}/{pl})]} & {[M, {}^{>6}\btau_{\leq10}(\Sigma^{-1}{g}/{pl})]} \\
	{[\Sigma M^{(6)},\btau_{\leq{10}}~pl/o]} & {[\sfrac{M}{M^{(6)}},\btau_{\leq{10}}~pl/o]} & {[M,\btau_{\leq{10}}~pl/o]} \\
	& {[\sfrac{M}{M^{(6)}},{}^{>6}\btau_{\leq{10}}~g/o]} & {[M,{}^{>6}\btau_{\leq{10}}~g/o]}
	\arrow[from=1-1, to=1-2]
	\arrow["{q^*}", two heads, from=1-2, to=1-3]
	\arrow["{j_*}", from=1-3, to=2-3]
	\arrow["{j_*}", from=1-2, to=2-2]
	\arrow["{q^*}", two heads, from=2-2, to=2-3]
	\arrow["j_*", two heads, from=1-1, to=2-1]
	\arrow[from=2-1, to=2-2]
	\arrow["{\psi_*}", from=2-2, to=3-2]
	\arrow["{q^*}", from=3-2, to=3-3]
	\arrow["{\psi_*}", from=2-3, to=3-3]
\end{tikzcd}
    \end{center}
whose columns and rows are part of exact sequences of abelian groups obtained from the fibration in \eqref{fib: g/pl to pl/o to g/o}, and the cofibration $M^{(6)}\hookrightarrow M \overset{q}{\lr}\sfrac{M}{M^{(6)}}$, respectively. Furthermore, the maps $j_*:[\Sigma M^{(6)}, {}^{>6}\btau_{\leq10}(\Sigma^{-1}{g}/{pl})]\to [\Sigma M^{(6)},\btau_{\leq 10}pl/o]$ and \\$q^*:[\sfrac{M}{M^{(6)}}, {}^{>6}\btau_{\leq10}(\Sigma^{-1}{g}/{pl})]\to [M, {}^{>6}\btau_{\leq10}(\Sigma^{-1}{g}/{pl})]$ are surjective. Thus, the Claim \ref{claim algebraic} holds for this commutative diagram, and hence we get the next proposition:

\begin{propn}\label{pro: claim used for mfld}
    If $y\not\in \im{j_*:[\sfrac{M}{M^{(6)}}, {}^{>6}\btau_{\leq10}(\Sigma^{-1}{g}/{pl})]\lr [\sfrac{M}{M^{(6)}}, \btau_{\leq 10}~pl/o]}$ then $q^*(y)\not\in \im{j_*:[M, {}^{>6}\btau_{\leq10}(\Sigma^{-1}{g}/{pl})]\lr [M, \btau_{\leq 10}~pl/o]}$ provided $q^*(y)\neq 0$.
\end{propn}


The map of our interest from \eqref{dig: commu sq for homo inertia grp of even dim mfld}, especially is $d^*:[\Sp{n},{}^{>6}\btau_{\leq{10}}~g/o]\lr [M,{}^{>6}\btau_{\leq{10}}~g/o] $. For that, the way we are going to use  \Cref{pro: claim used for mfld} is as follows. Consider the following diagram 
    \begin{equation}\label{dig: d^* + claim to M}
          \begin{tikzcd}[column sep=1.5em, row sep=2em]
 & {[\sfrac{M}{M^{(6)}}, {}^{>6}\btau_{\leq10}(\Sigma^{-1}{g}/{pl})]} & {[M, {}^{>6}\btau_{\leq10}(\Sigma^{-1}{g}/{pl})]} \\
	{[\Sp{n},\btau_{\leq{10}}~pl/o]} & {[\sfrac{M}{M^{(6)}},\btau_{\leq{10}}~pl/o]} & {[M,\btau_{\leq{10}}~pl/o]} \\
	{[\Sp{n},{}^{>6}\btau_{\leq{10}}~g/o]} & {[\sfrac{M}{M^{(6)}},{}^{>6}\btau_{\leq{10}}~g/o]} & {[M,{}^{>6}\btau_{\leq{10}}~g/o]}\\
    {} & {} & {}
	\arrow["{q^*}", two heads, from=1-2, to=1-3]
	\arrow["{j_*}", from=1-3, to=2-3]
	\arrow["{j_*}", from=1-2, to=2-2]
	\arrow["{q^*}", two heads, from=2-2, to=2-3]
	\arrow["{d^*}", from=2-1, to=2-2]
	\arrow["{\psi_*}", from=2-2, to=3-2]
	\arrow["{q^*}", from=3-2, to=3-3]
	\arrow["{\psi_*}", from=2-3, to=3-3]
   \arrow["{\psi_*}", from=2-1, to=3-1]
   \arrow["{d^*}", from=3-1, to=3-2]
   \arrow["{d^*}"', overlay, in=-165, out=-15, from=3-1, to=3-3]
\end{tikzcd}
    \end{equation}
whose columns are part of exact sequence using \eqref{fib: g/pl to pl/o to g/o}, the map $d^*$ is induced degree one maps, and $q^*$ is induced quotient map. For $n=8$, $9$ and $10$ the map $\psi_*:[\Sp{n},\btau_{\leq{10}}~pl/o]\lr [\Sp{n},{}^{>6}\btau_{\leq{10}}~g/o] $ is injective, surjective and isomorphism respectively. 
In most of the cases discussed below, we applied the \Cref{pro: claim used for mfld}. For that, we use diagrams \eqref{dig: appling claim to M} and \eqref{dig: d^* + claim to M} together, and show that if $0\neq y\in d^*([\Sp{n},\btau_{\leq{10}}~pl/o])\subseteq [\sfrac{M}{M^{(6)}},\btau_{\leq{10}}~pl/o]$, then $y\not\in \im{j_:[\sfrac{M}{M^{(6)}}, {}^{>6}\btau_{\leq10}(\Sigma^{-1}{g}/{pl})]\to [\sfrac{M}{M^{(6)}}, \btau_{\leq 10}~pl/o]}$.

\begin{thrm}\label{thm: Ih(M8)}
    Let $M^8$ be a closed manifold. Then $$I_h(M^8)=\{0\}.$$ 
\end{thrm}
\begin{proof}
In \eqref{dig: commu sq for homo inertia grp of even dim mfld}, as $[\Sp{8},PL/O]=\sfrac{\Z}{2}\lara{\epsilon}$ and  the map $\psi_*:[\Sp{8},PL/O]\lr [\Sp{8},G/O]$ is injective, it is enough to prove that the image of $\psi_*(\epsilon)$ under $d^*:[\Sp{8},G/O]\lr[M^8,G/O]$ is non-zero. 
By \Cref{thm: Ic8}, the map $d^*:[\Sp{8},\btau_{\leq8}~pl/o]\lr  [M^{8},\btau_{\leq{8}}~pl/o] $ is injective.
Recall that, ${}^{>6}\btau_{\leq8 }(\Sigma^{-1}{g}/{pl})_{(2)}\simeq \Sigma^7 H{\Z}_{(2)}$ and  $\btau_{\leq 8}(pl/o)\simeq \Sigma^7H\sfrac{\Z}{4}\vee \Sigma^8H\sfrac{\Z}{2}$, such that the map $\Sigma^7 H{\Z}_{(2)} \to \Sigma^8 H\sfrac{\Z}{2}$ is null homotopic. So we get that $d^*([\Sp{8},\btau_{\leq8}~pl/o])\bigcap \im{j_*:[{M^{8}},{}^{>6}\btau_{\leq8}(\Sigma^{-1}g/pl)]\lr [{M^{8}},\btau_{\leq8}~pl/o]}=\{0\}$. Therefore, in \eqref{dig: d^* + claim to M},  $\psi_*:[{M^{8}},\btau_{\leq8}~pl/o]\lr [{M^{8}},{}^{>6}\btau_{\leq8}~g/o]$ maps $d^*([\Sp{8},\btau_{\leq8}~pl/o])$ injectively. Using this in \eqref{dig: commu sq for homo inertia grp of even dim mfld}, completes the proof.\end{proof}


The following corollary uses the fact that every orientation preserving self-homotopy equivalence of $\rp{8}$ is homotopic to a diffeomorphism.
\begin{cor}
$I(\rp{8})=\sfrac{\Z}{2}.$
 \end{cor}

In the $9$-dimensional case, using \eqref{dig: d^* + claim to M} we first discuss the map $d^*:[\Sp{9},{}^{>6}\btau_{\leq9}~g/o]\to [\sfrac{M^{9}}{M^{(6)}},{}^{>6}\btau_{\leq{9}}~g/o]$ in the next lemma.

\begin{lemm}\label{lemm: 9-dim [S9,g/o] to [M/M6,g/o] results}
     Let $M^9$ be a closed oriented smooth manifold with $Sq^2\circ d_2:H^6(M^9;\sfrac{\Z}{4}) \lr H^9(M^9;\sfrac{\Z}{2})$.  
     \begin{enumerate}
         \item If $M^9$ is spin then $d^*:[\Sp{9},{}^{>6}\btau_{\leq{10}}~g/o]\lr [\sfrac{M^{9}}{M^{(6)}},{}^{>6}\btau_{\leq{10}}~g/o]$ is injective.

\item Suppose $M^9$ be non-spin and $Sq^2\circ d_2$ be trivial. 
    \begin{enumerate}
        \item  If $M^9$ satisfies \eqref{eq:M9/M6 "n" case homotopy decomposition} then $d^*:[\Sp{9},{}^{>6}\btau_{\leq{10}}~g/o]\lr [\sfrac{M^{9}}{M^{(6)}},{}^{>6}\btau_{\leq{10}}~g/o]$ is trivial.

         \item If $M^9$ satisfies \eqref{eq:M9/M6 "ion" case homotopy decomposition} then $d^*:[\Sp{9},{}^{>6}\btau_{\leq{10}}~g/o]\to[\sfrac{M^{9}}{M^{(6)}},{}^{>6}\btau_{\leq{10}}~g/o]$ has $\kr{d^*}=\sfrac{\Z}{2}\lara{\eta\circ\epsilon}$ and $\im{d^*}=\sfrac{\Z}{2}\lara{\mu}$. 
    \end{enumerate}
         
     \end{enumerate}
\end{lemm}
\begin{proof}
Recall that $[\Sp{9},{}^{>6}\btau_{\leq{10}}~g/o]\cong [\Sp{9},\Sigma^8\FF]\oplus[\Sp{9},\Sigma^7\FF_2]$. Additionally,\\ $[\Sp{9},\Sigma^8\FF]=\sfrac{\Z}{2}\lara{\eta\circ\epsilon}$ and $[\Sp{9},\Sigma^7\FF_2]=\sfrac{\Z}{2}\lara{\mu}$.
    \begin{enumerate}
      \item If $M^9$ is spin manifold then the $\sfrac{M^9}{M^{(6)}} $ decomposition \eqref{eq:M9/M6 spin-case homotopy decomposition} directly gives the injectivity of the map $d^*:[\Sp{9},{}^{>6}\btau_{\leq{10}}~g/o]\lr [\sfrac{M^{9}}{M^{(6)}},{}^{>6}\btau_{\leq{10}}~g/o]$.

\item Let $M^9$ be a non-spin manifold and $Sq^2\circ d_2=0$. Note that, from \eqref{eq:M9/M6 "n" case homotopy decomposition}, \eqref{eq:M9/M6 "ion" case homotopy decomposition}
the  degree map induces the following
\begin{center}\label{dig: S9 to M/M6 via X for g/o}
  \begin{tikzcd}[column sep=2em, row sep=1.2em]
	{[\Sp{9}, {}^{>6}\btau_{\leq10}~{g}/{o}]} && {[\sfrac{M^{9}}{M^{(6)}}, {}^{>6}\btau_{\leq10}~{g}/{o}]} \\
	& {[X, {}^{>6}\btau_{\leq10}~{g}/{o}]}
	\arrow["{d^*}", from=1-1, to=1-3]
	\arrow["{d^*}"', dashed, from=1-1, to=2-2]
	\arrow["{p^*}"', dashed, maps to, from=2-2, to=1-3]
\end{tikzcd}
\end{center}
where $X=C(\eta)$ or $C(\iota\circ\eta)$, and $p:\sfrac{M^9}{M^{(6)}}\lr X$ is the projection map. Now, we will discuss the image of $\lara{\mu} \in [\Sp{9}, \Sigma^7\FF_2] $ under $d^*:[\Sp{9}, {}^{>6}\btau_{\leq10}~{g}/{o}]\lr [X, {}^{>6}\btau_{\leq10}~{g}/{o}]$.
\begin{enumerate}
    \item If $M^9$ satisfies \eqref{eq:M9/M6 "n" case homotopy decomposition} then $X=C(\eta)$. Consider the induced long exact sequence due to the cofiber sequence for $\eta:\Sp{8}\lr \Sp{7}$
\begin{center}\label{seq: Cn g/o induced seq}
    \begin{tikzcd}
	\dots & {[\Sp{8}, {}^{>6}\btau_{\leq10}~{g}/{o}]} & {[\Sp{9}, {}^{>6}\btau_{\leq10}~{g}/{o}]} & {[C(\eta), {}^{>6}\btau_{\leq10}~{g}/{o}]} \\
	& {[\Sp{7}, {}^{>6}\btau_{\leq10}~{g}/{o}]} & {[\Sp{8}, {}^{>6}\btau_{\leq10}~{g}/{o}]}
	\arrow["{\eta^*}", from=1-2, to=1-3]
	\arrow["{\eta^*}", from=2-2, to=2-3]
	\arrow[from=1-3, to=1-4]
	\arrow[overlay, in=35, out=-155, looseness=0.4, from=1-4, to=2-2]
	\arrow[from=1-1, to=1-2]
\end{tikzcd}
\end{center}
Since $[\Sp{7}, {}^{>6}\btau_{\leq10}~{g}/{o}]=0$ and $\eta^*:{[\Sp{8}, {}^{>6}\btau_{\leq10}~{g}/{o}]}\lr {[\Sp{9}, {}^{>6}\btau_{\leq10}~{g}/{o}]}$ is surjective, gives 
$$[C(\eta), {}^{>6}\btau_{\leq10}~{g}/{o}]=0$$

Therefore, in \eqref{dig: S9 to M/M6 via X for g/o} the map $d^*:[\Sp{9}, \Sigma^7\FF_2]\lr [\sfrac{M^9}{M^{(6)}}, {}^{>6}\btau_{\leq10}~{g}/{o}]$ is trivial.


\item If $M^9$ satisfies \eqref{eq:M9/M6 "ion" case homotopy decomposition} then $X=C(\iota\circ\eta)$.
Now consider the induced long exact sequence from the cofiber sequence for $\iota\circ\eta:\Sp{8}\lr \Sp{7}\lr M(\sfrac{\Z}{2^r},7)$
\begin{center}
   \begin{tikzcd}
	& {[\Sp{8}, {}^{>6}\btau_{\leq10}~g/o]} \\
	{[\Sigma M(\sfrac{\Z}{2^r},7), {}^{>6}\btau_{\leq10}~g/o]} & {[\Sp{9}, {}^{>6}\btau_{\leq10}~g/o]} & {[C(\iota\circ\eta),{}^{>6}\btau_{\leq10}~g/o]} \\
	{[M(\sfrac{\Z}{2^r},7), {}^{>6}\btau_{\leq10}~g/o]} & {[\Sp{8}, {}^{>6}\btau_{\leq10}~g/o]} \\
	& {[\Sp{7}, {}^{>6}\btau_{\leq10}~g/o]}
	\arrow["{d^*}", from=2-2, to=2-3]
	\arrow[from=2-1, to=2-2]
	\arrow[dashed, from=3-1, to=4-2]
	\arrow["{\eta^*}"', dashed, from=4-2, to=3-2]
	\arrow[overlay, in=35, out=-155, looseness=0.4, from=2-3, to=3-1]
	\arrow[from=3-1, to=3-2]
	\arrow[dashed, from=2-1, to=1-2]
	\arrow["{\eta^*}", dashed, from=1-2, to=2-2]
\end{tikzcd}
\end{center}
where 
$\im{[\Sigma M(\sfrac{\Z}{2^r},7), {}^{>6}\btau_{\leq10}~g/o]\lr [\Sp{8}, {}^{>6}\btau_{\leq10}~g/o]}\cong\sfrac{\Z}{2}\lara{\epsilon}\subseteq \sfrac{\Z}{2}\lara{\ol{\nu}}\oplus\sfrac{\Z}{2}\lara{\epsilon}$.\\
This implies 
\begin{equation}
    \kr{d^*:[\Sp{9}, {}^{>6}\btau_{\leq10}~g/o]\lr [C(\iota\circ\eta), {}^{>6}\btau_{\leq10}~g/o]}=\sfrac{\Z}{2}\lara{\eta\circ\epsilon}
\end{equation}and 
\begin{equation}
    \im{d^*:[\Sp{9}, {}^{>6}\btau_{\leq10}~g/o]\lr [C(\iota\circ\eta), {}^{>6}\btau_{\leq10}~g/o]}=\sfrac{\Z}{2}\lara{\mu}.
\end{equation}
\end{enumerate} 
    \end{enumerate}
    This completes the proof in all cases.
\end{proof}
\begin{thrm}\label{thm: S9 to M at g/o level maps}
      Let $M^9$ be a closed oriented smooth  manifold with $Sq^2\circ d_2:H^6(M^9;\sfrac{\Z}{4}) \to H^9(M^9;\sfrac{\Z}{2})$. 
      \begin{enumerate}
          \item If $M^9$ is spin then the map $d^*:[\Sp{9}, {}^{>6}\btau_{\leq{10}}~g/o]\lr [{M^9}, {}^{>6}\btau_{\leq{10}}~g/o]$ is injective.

   \item Let $M^9$ be non-spin and the map $Sq^2\circ d_2$ be trivial.
   \begin{enumerate}
        \item If $M^9$ satisfies \eqref{eq:M9/M6 "n" case homotopy decomposition} then $d^*:[\Sp{9}, {}^{>6}\btau_{\leq{10}}~g/o]\lr [{M^9}, {}^{>6}\btau_{\leq{10}}~g/o]$ maps $\lara{\mu}$ to $0$.
        
        \item If $M^9$ satisfies \eqref{eq:M9/M6 "ion" case homotopy decomposition} then $d^*:[\Sp{9}, {}^{>6}\btau_{\leq{10}}~g/o]\lr [{M^9}, {}^{>6}\btau_{\leq{10}}~g/o]$ maps $\lara{\mu}$ injectively.
    \end{enumerate}
      \end{enumerate}
\end{thrm}

\begin{proof}
\begin{enumerate}
   \item Let $M^9$ be a spin manifold. Applying \Cref{pro: claim used for mfld}, using Lemma \ref{lemm: 9-dim [S9,g/o] to [M/M6,g/o] results}$(1)$ and \Cref{thm: Ic(9)}$(1)$ together, we obtain the map $d^*:[\Sp{9},{}^{>6}\btau_{\leq{10}}~g/o]\lr [M^{9},{}^{>6}\btau_{\leq{10}}~g/o]$ is injective. 
    \item Let $M^9$ be non-spin and the map $Sq^2\circ d_2$ be trivial.
\begin{enumerate}
   \item If $M^9$ satisfies \eqref{eq:M9/M6 "n" case homotopy decomposition}, then from Lemma \ref{lemm: 9-dim [S9,g/o] to [M/M6,g/o] results}$(2)(a)$ we know the map $d^*:[\Sp{9}, {}^{>6}\btau_{\leq{10}}~g/o]\to [\sfrac{M^9}{M^{(6)}}, {}^{>6}\btau_{\leq{10}}~g/o]$ is trivial. Since, the degree one map $d^*:[\Sp{9}, {}^{>6}\btau_{\leq{10}}~g/o]\to [{M^9}, {}^{>6}\btau_{\leq{10}}~g/o]$ factors through $[\sfrac{M^9}{M^{(6)}}, {}^{>6}\btau_{\leq{10}}~g/o]$, makes the statement clear in this case.

  \item Now consider the case when $M^9$ satisfies \eqref{eq:M9/M6 "ion" case homotopy decomposition}. 
From Lemma \ref{lemm: 9-dim [S9,g/o] to [M/M6,g/o] results}$(2)(b)$ it follows that $\sfrac{\Z}{2}\lara{\mu}$ is in the $\im{d^*:[\Sp{9}, {}^{>6}\btau_{\leq{10}}~g/o]\lr [\sfrac{M^9}{M^{(6)}}, {}^{>6}\btau_{\leq{10}}~g/o]}$. As, under $\psi_*:[\Sp{9}, \btau_{\leq{10}}~pl/o]\lr [\Sp{9}, {}^{>6}\btau_{\leq{10}}~g/o]$ the group $\sfrac{\Z}{2}\lara{\mu} $ is mapped injectively, implies in \eqref{dig: d^* + claim to M}, $d^*(\mu)\not\in \im{j_*:[\sfrac{M^9}{M^{(6)}}, {}^{>6}\btau_{\leq{10}}\Sigma^{-1}(g/pl)]\lr [\sfrac{M^9}{M^{(6)}}, \btau_{\leq{10}}~pl/o]}$.
Also, it follows from \Cref{thm: Ic(9)}$(2)(a)$ that $d^*(\mu)$ is non-zero in $[{M^9}, \btau_{\leq{10}}~pl/o]$. Finally using  \Cref{pro: claim used for mfld}, we get that $d^*(\mu)$ gets mapped non-trivially under $\psi_*:[{M^9}, \btau_{\leq{10}}~pl/o]\lr [{M^9}, {}^{>6}\btau_{\leq{10}}~g/o]$. Therefore, $\mu$ gets mapped non-trivially under $d^*:[\Sp{9}, {}^{>6}\btau_{\leq{10}}~g/o]\lr [{M^9}, {}^{>6}\btau_{\leq{10}}~g/o]$ as well. Hence the statement.
\end{enumerate}\end{enumerate}\end{proof}

\begin{thrm}\label{thm: Ih(9)}
    Let $M^9$ be a closed oriented smooth manifold with the map $Sq^2\circ d_2:H^6(M^9;\sfrac{\Z}{4}) \lr H^9(M^9;\sfrac{\Z}{2})$.
    \begin{enumerate}
     \item If $M^9$ is a spin manifold then $$I_h(M^9)\bigcap \faktor{\Th{9}}{bP_{10}}=\{0\}.$$
    
     \item If $M^9$ is a spin simply-connected manifold then $$I_h(M^9)=\{0\}.$$ 

     \item Suppose $M^9$ is non-spin and the map $Sq^2\circ d_2=0$.
        \begin{enumerate}
         \item If $M^9$ satisfies \ref{eq:M9/M6 "n" case homotopy decomposition} then
                $$I_h(M^9)\supseteq \sfrac{\Z}{2}\lara{\eta\circ\epsilon}\oplus\sfrac{\Z}{2}\lara{\mu}.$$

        \item   If $M^9$ satisfies \ref{eq:M9/M6 "ion" case homotopy decomposition} then
               $$I_h(M^9)\supseteq\sfrac{\Z}{2}\lara{\eta\circ\epsilon} \text{ and } I_h(M^9)\nsupseteq \sfrac{\Z}{2}\lara{\mu}.$$
                
           \end{enumerate}
         \item If $M^9$ is non-spin with $Sq^2\circ d_2\neq0$ then $$I_h(M^9)\supseteq \sfrac{\Z}{2}\lara{\eta\circ\epsilon}\oplus\sfrac{\Z}{2}\lara{\mu}.$$
      
    \end{enumerate}
\end{thrm}
\begin{proof}    Recall that, $I_h(M^n)=\kr{f_{M^n}:[\Sp{n}, PL/O]\lr S^{\text{Diff}}(M^n)}$, where $f_{M^n}$ is part of the commutative square \eqref{dig: commu sq for homo inertia grp of even dim mfld}.
    \begin{enumerate}
 
        \item Let $M^9$ be a spin manifold. Recollect that, $\sfrac{\Th{9}}{bP_{10}}=\sfrac{\Z}{2}\lara{\eta\circ\epsilon}\oplus\sfrac{\Z}{2}\lara{\mu}\subseteq [\Sp{9}, \btau_{\leq{10}}~pl/o]$, and in \eqref{dig: commu sq for homo inertia grp of even dim mfld} the map $\psi_*:[\Sp{9}, \btau_{\leq{10}}~pl/o]\lr [\Sp{9}, {}^{>6}\btau_{\leq{10}}~g/o]$ maps $\sfrac{\Z}{2}\lara{\eta\circ\epsilon}\oplus\sfrac{\Z}{2}\lara{\mu}$ injectively.
         Furthermore, by \Cref{thm: S9 to M at g/o level maps}$(1)$, the map $d^*:[\Sp{9}, {}^{>6}\btau_{\leq{10}}~g/o]\lr [{M^9}, {}^{>6}\btau_{\leq{10}}~g/o]$ is injective, thereby making the statement clear in this case.
         
     \item If $M^9$ is simply-connected spin manifold then the result \cite[Proposition II.2]{Brumfiel1971HomotopyEOUnpublished} says that $I_h(M^9)\subseteq \ckr{J_9}=\sfrac{\Z}{2}\lara{\eta\circ\epsilon}\oplus\sfrac{\Z}{2}\lara{\mu}$. Hence, from statement $(1)$ we get $I_h(M^9)=0$.

     \item Let $M^9$ be a non-spin manifold and the map $Sq^2\circ d_2=0$.
\begin{enumerate}
    \item Suppose $M^9$ satisfies \ref{eq:M9/M6 "n" case homotopy decomposition}. Consider the following commutative diagram
\begin{center}
  \begin{tikzcd}[column sep=3em]
	& {[\Sp{9},PL/O]} \\
	{[\Sigma M_0^{9}, G/O]} & {[\Sp{9}, G/O]} & {[M^{9}, G/O]}
	\arrow["\textbf{d}", from=2-1, to=1-2]
	\arrow["{\psi_*}", from=1-2, to=2-2]
	\arrow["{d^*}"', from=2-2, to=2-3]
	\arrow["{(\Sigma h)^*}"', from=2-1, to=2-2]
\end{tikzcd}
\end{center}
where $M^9_0=M^9\setminus \mathit{int} (\D{9})$, and $\textbf{d}:[\Sigma M^9_0,G/O]\lr [\Sp{9},PL/O]$ is the map given in \cite[Proposition 3.1]{Brumfiel1971HomotopyEOUnpublished}, such that $I_h(M^9)=\im{\textbf{d}}$. The bottom row is obtained from cofiber sequence ${M^9}\overset{d}{\lr}\Sp{9}\overset{\Sigma h}{\lr}\Sigma M_0^9$ such that, $h$ is the top cell attaching map of $M^9$ and $d$ is the degree one map.

Using \Cref{thm: Ic(9)}$(2)(a)$ and \Cref{thm: S9 to M at g/o level maps}$(1)(a)$, together gives $d^*:[\Sp{9}, G/O]\lr [M^{9}, G/O]$ is trivial. Hence, in the above diagram, $(\Sigma h)^*:[\Sigma M_0^{9}, G/O]\lr[\Sp{9}, G/O]$ is surjective. As $\sfrac{\Z}{2}\lara{\eta\circ\epsilon}\oplus\sfrac{\Z}{2}\lara{\mu}\subseteq[\Sp{9}, PL/O]$ makes the statement clear in this case.

\item If $M^9$ satisfies \ref{eq:M9/M6 "ion" case homotopy decomposition}, then by \Cref{thm: Ic(9)}$(2)(a)$ we know that $\sfrac{\Z}{2}\lara{\eta\circ\mu}\subseteq I_h(M^9)$.  Now, in \eqref{dig: commu sq for homo inertia grp of even dim mfld} recall that, $\psi_*(\mu)$ is non-zero in $[\Sp{9}, {}^{>6}\btau_{\leq{10}}~g/o]$, and the \Cref{thm: S9 to M at g/o level maps}$(2)(b)$ shows that, $\sfrac{\Z}{2}\lara{\mu}$ is mapped injectively under the map $d^*:[\Sp{9}, {}^{>6}\btau_{\leq{10}}~g/o]\lr [M^{9}, {}^{>6}\btau_{\leq{10}}~g/o]$. Therefore, $f_{M^9}(\mu)$ is also non-zero, which complete the proof for this case.
    \end{enumerate}
    
     \item The statement for the case when $M^9$ is non-spin and the map $Sq^2\circ d_2\neq0$ is clear from \Cref{thm: Ic(9)}$(2)(b)$ since $ \sfrac{\Z}{2}\lara{\eta\circ\epsilon}\oplus\sfrac{\Z}{2}\lara{\mu}=I_c(M^9)\subseteq I_h(M^9)$.
\end{enumerate}\end{proof}

Recall that, $L^9(m)$ is a closed oriented smooth $9$-manifold, and its inertia group, which depends on $m$, is discussed below.

\begin{thrm}\label{thm:lensspace}
    Let $m$ be a positive integer and $n$ be a non-negative integer. 
    \begin{enumerate}
        \item If $m=2n+1$ then $$I(L^9(m))=\{0\}.$$
    
       \item If $m=4n+2$ then $$I(L^9(m))=\sfrac{\Z}{2}\lara{\eta\circ\epsilon}.$$

        \item If $m=4n$ then $$I(L^9(m))=\sfrac{\Z}{2}\lara{\eta\circ\epsilon}\opl{} bP_{10}.$$
    \end{enumerate}
\end{thrm}
\begin{proof}
 Note that, if $m$ is odd then $L^9(m)$ is a spin manifold; otherwise, it is non-spin. Observe that, if $L^9(m)$ is non-spin, then $Sq^2\circ d_2=0$, and $L^9(m)$ satisfies \eqref{eq:M9/M6 "ion" case homotopy decomposition}. Additionally, $bP_{10}\subseteq I_h(L^9(m))$ if and only if $4|m$ by \cite[Theorem 4.2]{KervaireObstruction}. Combining this together with the fact that; an orientation-preserving self-homotopy equivalence of a lens space is homotopic to the identity, and using \Cref{thm: Ih(9)}, completes the proof.
\end{proof}


        

Note that when $m=2$, $L^{2n+1}(m)$ is nothing but the usual real projective space $\rp{2n+1}$. The following remark is part of \Cref{thm: Ih(9)} and serves as a special case of $m=2$ in \Cref{thm:lensspace}.
\begin{rem}
    There exists an unique exotic sphere $\Sigma\in \Th{9}$ such that $\rp{9}\csum\Sigma$ is diffeomorphic to $\rp{9}$.
\end{rem}

Now, in the $10$-dimensional case for the homotopy inertia group we just need to check the case when $M^{10}$ is spin and $\Phi=0$. For the same, recall the homotopy decompositions of $\sfrac{M^{10}}{M^{(6)}}$ \eqref{eq: 10-dim-spin "f" null homotopic case}, \eqref{eq: 10-dim-spin "f" n2 case} and \eqref{eq: 10-dim-spin "f" ion2 case}.
Since $\pi_{10}(pl/o)\cong \pi_{10}(g/o)\cong\Z_2\lara{\eta\circ\mu}\oplus\Z_3\lara{\beta_1}$, it suffices to work locally at prime 2 and 3. Let us first discuss the $3$-localized case.
\begin{thrm}\label{thm: 3-localized S10 to M10 at g/o}
     Let $ M^{10} $ be a closed orientated smooth manifold. Then the induced degree one map $d^*:[\Sp{10},{}^{>6}\btau_{\leq 10}~g/o]_{(3)}\lr [{M^{10}},{}^{>6}\btau_{\leq 10}~g/o]_{(3)}$ is injective.
\end{thrm}
\begin{proof} 
Recall that $g/o_{(3)}\simeq \sckr{J}_{(3)}\times bso_{(3)}$. Since $[\Sp{10},bso]_{(3)}=0$; we get \\$[\Sp{10},{}^{>6}\btau_{\leq 10}~g/o]_{(3)}=[\Sp{10},\sckr{J}]_{(3)}$. Note that $$\pi_i(\sckr{J})_{(3)}\begin{cases}
    0 & \text{ if } i\leq 9\\
    \sfrac{\Z}{3} &\text{ if }  i=10
\end{cases}$$
This implies $\btau_{\leq 10}~\sckr{J}_{(3)}=K(\sfrac{\Z}{3},10)$. Therefore the statement is true because $d^*:H^{10}(\Sp{10};\sfrac{\Z}{3})\lr H^{10}(M^{10};\sfrac{\Z}{3})$ is injective.\end{proof}

Now let us work $2$-locally.
\begin{thrm} \label{thm: 2-localized S10 to M10 at g/o}
  Let $ M^{10} $ be a closed  smooth $ 10 $-manifold with $H_1(M)=0$. Then the induced degree one map 
    $d^*:[\Sp{10}, {}^{>6}\btau_{\leq{10}}~g/o]_{(2)}\lr [{M^{10}}, {}^{>6}\btau_{\leq{10}}~g/o]_{(2)}$ is injective if $M^{10}$ satisfy \eqref{eq: 10-dim-spin "f" null homotopic case} or \eqref{eq: 10-dim-spin "f" ion2 case}, and is trivial if $M^{10}$ satisfy \eqref{eq: 10-dim-spin "f" n2 case}.
\end{thrm}
\begin{proof} Consider the case when $M^{10}$ satisfies \eqref{eq: 10-dim-spin "f" null homotopic case}. Since $[\Sp{10}, {}^{>6}\btau_{\leq{10}}(\Sigma^{-1}g/pl)]_{(2)}=0$,\\ $d^*([\Sp{10}, \btau_{\leq{10}}~pl/o]_{(2)})\bigcap j_*([\sfrac{M^{10}}{M^{(6)}}, {}^{>6}\btau_{\leq{10}}(\Sigma^{-1}g/pl)]_{(2)})=\{0\}$. From \Cref{thm: 10-dim Ic}, we know that $d^*(\eta\circ\mu)$ is non-zero in $[M^{10}, \btau_{\leq{10}}~pl/o]$. Therefore, by \Cref{pro: claim used for mfld} the map \\$d^*:[\Sp{10}, {}^{>6}\btau_{\leq{10}}~g/o]_{(2)}\lr [{M^{10}}, {}^{>6}\btau_{\leq{10}}~g/o]_{(2)}$ is injective.

    For the remaining cases, the homotopy decompositions of $\sfrac{M^{10}}{M^{(6)}}$ in \eqref{eq: 10-dim-spin "f" n2 case} or \eqref{eq: 10-dim-spin "f" ion2 case}, gives the following commutative diagram
\begin{center}\label{dig: S10 to M/M6 via X for g/o}
   \begin{tikzcd}[column sep=2em, row sep=1.2em]
	{[\Sp{10}, {}^{>6}\btau_{\leq10}~{g}/{o}]} && {[\sfrac{M^{10}}{M^{(6)}}, {}^{>6}\btau_{\leq10}~{g}/{o}]} \\
	& {[X, {}^{>6}\btau_{\leq10}~{g}/{o}]}
	\arrow["{d^*}", from=1-1, to=1-3]
	\arrow["{d^*}"', dashed, from=1-1, to=2-2]
	\arrow["{p^*}"', dashed, maps to, from=2-2, to=1-3]
\end{tikzcd}
\end{center}
where $X=C(\eta^2)$ or $C(\iota\circ\eta^2)$, and $p:\sfrac{M^{10}}{M^{(6)}}\lr X$ is the projection map. According to our interest for homotopy inertia group, as $p^*$ is injective, so first we will check the image of $\lara{\eta\circ\mu} \in [\Sp{{10}}, \Sigma^7\EE] $ under $d^*:[\Sp{{10}}, {}^{>6}\btau_{\leq10}~{g}/{o}]\lr [X, {}^{>6}\btau_{\leq10}~{g}/{o}]$.

Now, suppose $M^{10}$ satisfies \eqref{eq: 10-dim-spin "f" n2 case}. 
Consider the cofiber sequence for $\eta^2:\Sp{9}\lr \Sp{7}$
\begin{center}\label{seq: cn^2 g/o induced seq}
     \begin{tikzcd}
	 {[\Sp{8}, {}^{>6}\btau_{\leq10}~{g}/{o}]_{(2)}} & {[\Sp{10}, {}^{>6}\btau_{\leq10}~{g}/{o}]_{(2)}} & {[C(\eta^2), {}^{>6}\btau_{\leq10}~{g}/{o}]_{(2)}} \\
	{[\Sp{7}, {}^{>6}\btau_{\leq10}~{g}/{o}]_{(2)}} & {[\Sp{9}, {}^{>6}\btau_{\leq10}~{g}/{o}]_{(2)}}
	\arrow["{(\eta^2)^*}", from=1-1, to=1-2]
	\arrow["{(\eta^2)^*}", from=2-1, to=2-2]
	\arrow[from=1-2, to=1-3]
	\arrow[overlay, in=35, out=-155, looseness=0.4, from=1-3, to=2-1]
\end{tikzcd}
\end{center}
Here, the group $[\Sp{7}, {}^{>6}\btau_{\leq10}~{g}/{o}]_{(2)}=0$. Further, we have
\begin{center}
    \begin{tikzcd}[column sep=tiny, row sep=1.6em]
	{[\Sp{8}, {}^{>6}\btau_{\leq10}~{g}/{o}]_{(2)}} & {[\Sp{8},\sckr{J}]_{(2)}} & {[\Sp{8},bso]_{(2)}} \\
	{[\Sp{10}, {}^{>6}\btau_{\leq10}~{g}/{o}]_{(2)}} & {[\Sp{10},\sckr{J}]_{(2)}} & {[\Sp{10},bso]_{(2)}}
	\arrow["{(\eta^2)^*}", from=1-1, to=2-1]
	\arrow["{(\eta^2)^*}", from=1-2, to=2-2]
	\arrow["\bigoplus"{description}, draw=none, from=1-2, to=1-3]
	\arrow["{(\eta^2)^*}", from=1-3, to=2-3]
	\arrow["\bigoplus"{description}, draw=none, from=2-2, to=2-3]
	\arrow["\cong"{description}, draw=none, from=1-1, to=1-2]
	\arrow["\cong"{description}, draw=none, from=2-1, to=2-2]
\end{tikzcd}
\end{center}
where $[\Sp{10},\sckr{J}]_{(2)}=0$ and $(\eta^2)^*:[\Sp{8},bso]_{(2)}\lr [\Sp{10},bso]_{(2)}$ is non-zero \cite[Proposition 7.1]{Adams}. 
Therefore, the map $(\eta^2)^*:{[\Sp{8}, {}^{>6}\btau_{\leq10}~{g}/{o}]_{(2)}}\lr {[\Sp{10}, {}^{>6}\btau_{\leq10}~{g}/{o}]_{(2)}}$ is surjective, proving that in the exact sequence \ref{seq: cn^2 g/o induced seq}
\begin{equation}
    [C(\eta^2), {}^{>6}\btau_{\leq10}~{g}/{o}]_{(2)}=0.
\end{equation}
This gives in \eqref{dig: S10 to M/M6 via X for g/o} the map $d^*:[\Sp{10}, {}^{>6}\btau_{\leq10}~{g}/{o}]_{(2)}\lr [\sfrac{M^{10}}{M^{(6)}}, {}^{>6}\btau_{\leq10}~{g}/{o}]_{(2)} $ is trivial. As a result, in \eqref{dig: d^* + claim to M} the map $d^*:[\Sp{10}, {}^{>6}\btau_{\leq10}~{g}/{o}]_{(2)}\lr [M^{10}, {}^{>6}\btau_{\leq10}~{g}/{o}]_{(2)} $ is trivial.

For the case when $M^{10}$ satisfies \eqref{eq: 10-dim-spin "f" ion2 case}, consider the exact sequence 
\begin{center}
   \begin{tikzcd}[column sep=1em,]
	{[\Sigma M(\sfrac{\Z}{2^r},7), {}^{>6}\btau_{\leq10}(\Sigma^{-1}{g}/{pl})]} & {[\Sp{10}, {}^{>6}\btau_{\leq10}(\Sigma^{-1}{g}/{pl})]} & {[C(\iota\circ\eta^2),{}^{>6}\btau_{\leq10}(\Sigma^{-1}{g}/{pl})]} \\
	{[M(\sfrac{\Z}{2^r},7), {}^{>6}\btau_{\leq10}(\Sigma^{-1}{g}/{pl})]} && {[\Sp{9}, {}^{>6}\btau_{\leq10}(\Sigma^{-1}{g}/{pl})]} \\[-1.3em]
	& {[\Sp{7}, {}^{>6}\btau_{\leq10}(\Sigma^{-1}{g}/{pl})]}
	\arrow["{d^*}", from=1-2, to=1-3]
	\arrow[from=1-1, to=1-2]
	\arrow[from=2-1, to=2-3]
	\arrow[overlay, in=35, out=-155, looseness=0.4, from=1-3, to=2-1]
	\arrow[dashed, in=175, out=-30, looseness=0.4, dashed, from=2-1, to=3-2]
	\arrow["{(\eta^2)^*}"', in=-145, out=5, looseness=0.4, dashed, from=3-2, to=2-3]
\end{tikzcd}
\end{center}
where $[\Sp{10}, {}^{>6}\btau_{\leq10}(\Sigma^{-1}{g}/{pl})]_{(2)}=0$, and the map $(\eta^2)^*:[\Sp{7}, {}^{>6}\btau_{\leq10}(\Sigma^{-1}{g}/{pl})]\to [\Sp{9}, {}^{>6}\btau_{\leq10}(\Sigma^{-1}{g}/{pl})]$ is trivial. Thus, $$[C(\iota\circ\eta^2), {}^{>6}\btau_{\leq10}(\Sigma^{-1}{g}/{pl})]\cong[M(\sfrac{\Z}{2^r},7), {}^{>6}\btau_{\leq10}(\Sigma^{-1}{g}/{pl})]$$

The cofiber sequence for $M(\sfrac{\Z}{2^r},7)$
induces the following 
\begin{center}
  \begin{tikzcd}[column sep=1em]
	{[\Sp{8}, {}^{>6}\btau_{\leq10}(\Sigma^{-1}{g}/{pl})]} & {[\Sp{8}, {}^{>6}\btau_{\leq10}(\Sigma^{-1}{g}/{pl})]} & {[M(\sfrac{\Z}{2^r},7), {}^{>6}\btau_{\leq10}(\Sigma^{-1}{g}/{pl})]} \\
	 {[\Sp{7}, {}^{>6}\btau_{\leq10}(\Sigma^{-1}{g}/{pl})]} & {[\Sp{7}, {}^{>6}\btau_{\leq10}(\Sigma^{-1}{g}/{pl})]}
	\arrow["{\times 2^r}", from=2-1, to=2-2]
	\arrow[overlay, in=35, out=-155, looseness=0.4, from=1-3, to=2-1]
	\arrow["{\times 2^r}", from=1-1, to=1-2]
	\arrow[from=1-2, to=1-3]
\end{tikzcd}
\end{center}
Note that $[\Sp{8}, {}^{>6}\btau_{\leq10}(\Sigma^{-1}{g}/{pl})]=0$ and $[\Sp{7}, {}^{>6}\btau_{\leq10}(\Sigma^{-1}{g}/{pl})]\cong\Z$ together implies\\  $[M(\sfrac{\Z}{2^r},7), {}^{>6}\btau_{\leq10}(\Sigma^{-1}{g}/{pl})]=0$, which gives,
\begin{equation}
    [C(\iota\circ\eta^2), {}^{>6}\btau_{\leq10}(\Sigma^{-1}{g}/{pl})]=0
\end{equation}
Therefore, it follows from \eqref{fib: g/pl to pl/o to g/o} that the map
\begin{equation}\label{eq: c(in2) pl/o to g/o inj}
   \psi_*: [C(\iota\circ\eta^2), \btau_{\leq 10}~pl/o] \lr [C(\iota\circ\eta^2), {}^{>6}\btau_{\leq 10}~g/o]
\end{equation}
is injective. Consequently, in \eqref{dig: d^* + claim to M}, $d^*(\eta\circ\mu)\not\in j_*([\sfrac{M^{10}}{M^{(6)}}, {}^{>6}\btau_{\leq 10}\Sigma^{-1}(g/pl)])$. Since, according to \Cref{thm: 10-dim Ic}$(1)(a)$, the element $d^*(\eta\circ\mu)$ is non-zero in $[{M^{10}}, \btau_{\leq 10}~pl/o]$, and by \Cref{pro: claim used for mfld}, the element $d^*(\eta\circ\mu)$ gets mapped non-trivially under $\psi_*:[{M^{10}}, \btau_{\leq 10}~pl/o]\lr [{M^{10}}, {}^{>6}\btau_{\leq 10}~g/o]$. Therefore, the map $d^*:[\Sp{10}, {}^{>6}\btau_{\leq 10}~g/o]\lr [{M^{10}}, {}^{>6}\btau_{\leq 10}~g/o]$ maps $\eta\circ\mu$ non-trivially, completing the proof.

\end{proof}

\begin{thrm}\label{thm: 10 Ih}
     Let $ M^{10} $ be a closed simply-connected smooth $ 10 $-manifold.
     \begin{enumerate}
         \item If $M^{10}$ is a spin manifold and $\Phi=0$ then\\
         $$I_h(M^{10})=\begin{cases}
             0  & \text{ if } M^{10} \text{ satisfies \eqref{eq: 10-dim-spin "f" null homotopic case} or \eqref{eq: 10-dim-spin "f" ion2 case},}\\
             \sfrac{\Z}{2}\lara{\eta\circ\mu} & \text{ if } M^{10} \text{ satisfies \eqref{eq: 10-dim-spin "f" n2 case}}.
         \end{cases}$$

         \item If $M^{10}$ is a spin manifold and $\Phi\neq 0$ or is a non-spin manifold then $$I_h(M^{10})=\sfrac{\Z}{2}\lara{\eta\circ\mu}.$$

     \end{enumerate}
\end{thrm}
\begin{proof}
    The proof follows from \eqref{dig: commu sq for homo inertia grp of even dim mfld} using for the simply-connected $10$-manifolds together with \Cref{thm: 3-localized S10 to M10 at g/o} and \Cref{thm: 2-localized S10 to M10 at g/o}.
\end{proof}

\begin{rem}\label{remark: Ih of 10-mfld orientated non-oreinted (3)-localized}
    \begin{enumerate}[(i)]
        \item The \Cref{thm: 3-localized S10 to M10 at g/o} gives more general result: Let $ M^{10} $ be a closed orientated smooth manifold. Then $I_h(M^{10})\bigcap~({\Th{10}})_{(3)}=\{0\}$.

   \item  Note that, if $M^{10}$ is non-oriented then $H^{10}(M^{10};\sfrac{\Z}{3})=0$; therefore the similar proof of \Cref{thm: 3-localized S10 to M10 at g/o} gives $I_h(M^{10})\supseteq \sfrac{\Z}{3}\lara{\beta_1}$.
    \end{enumerate}
\end{rem}

\subsection{Inertia group of $\rp{10}$}

Now, we compute the homotopy inertia group of a non oriented manifold $\rp{10}$ by considering the proof techniques used for \Cref{thm: 10 Ih} with minimal adjustments.

\begin{thrm}\label{thm: rp10}
$I_h(\rp{10})=\sfrac{\Z}{3}\lara{\beta_1}\oplus\sfrac{\Z}{2}\lara{\eta\circ\mu}$.
\end{thrm}
\begin{proof} 
For $I_h(\rp{10})$, as mentioned in remark \ref{remark: Ih of 10-mfld orientated non-oreinted (3)-localized}$(ii)$, we have $\sfrac{\Z}{3}\lara{\beta_1}\subseteq I_h(\rp{10})$. Therefore, it is enough to work $2$-locally.
Consider the diagram \eqref{dig: commu sq for homo inertia grp of even dim mfld}, in which note that ${\psi_*}:[\Sp{10},PL/O]\lr [\Sp{10},G/O]$ and $g':S^{\text{Diff}}(M^n)\lr [M^n,{G}/{O}]$ are injective. So it is enough to show that the map $d^*:[\Sp{10},{}^{>6}\btau_{\leq 10}~g/o]_{(2)}\lr [\rp{10},{}^{>6}\btau_{\leq 10}~g/o]_{(2)}$ is trivial. Note that, this map $d^*$ factors through $[\sfrac{\rp{10}}{\rp{6}},{}^{>6}\btau_{\leq 10}~g/o]$, as shown in the diagram below,
\begin{equation}\label{dig: degree one maps for rp10/7 to 10/6 to rp10 for g/o}
  \begin{tikzcd}[column sep=2em, row sep=1.2em]
	{[\Sp{10}, {}^{>6}\btau_{\leq10}~{g}/{o}]} && {[\rp{10},{}^{>6}\btau_{\leq 10}~g/o]} \\
	& {[\sfrac{\rp{10}}{\rp{6}},{}^{>6}\btau_{\leq 10}~g/o]}
	\arrow["{d^*}", from=1-1, to=1-3]
	\arrow["{d^*}"', dashed, from=1-1, to=2-2]
	\arrow["{q^*}"', dashed, maps to, from=2-2, to=1-3]
\end{tikzcd}
\end{equation}
where $q:\rp{10}\lr \sfrac{\rp{10}}{\rp{6}}$ is the quotient map.
To complete the proof, we claim that the map $d^*:[\Sp{10},{}^{>6}\btau_{\leq 10}~g/o]\lr [\sfrac{\rp{10}}{\rp{6}},{}^{>6}\btau_{\leq 10}~g/o]$ is trivial.

 As $\sfrac{\rp{10}}{\rp{7}}\simeq \Sp{8}\vee M(\sfrac{\Z}{2},9)$ and $\Sigma\sfrac{\rp{7}}{\rp{6}}\simeq\Sp{8}$, consider the following commutative diagram
\begin{center}\label{dig: Rp7/6 to Rp10/6 to Rp10/7 at g/o}
   \begin{tikzcd}[column sep=small]
	{[\Sigma\sfrac{\rp{7}}{\rp{6}},{}^{>6}\btau_{\leq 10}~g/o]} & {[\sfrac{\rp{10}}{\rp{7}},{}^{>6}\btau_{\leq 10}~g/o]} & {[\sfrac{\rp{10}}{\rp{6}},{}^{>6}\btau_{\leq 10}~g/o]} \\
	{[\Sp{8},{}^{>6}\btau_{\leq 10}~g/o]} & {[\Sp{8}\vee M(\sfrac{\Z}{2},9),{}^{>6}\btau_{\leq 10}~g/o]} & {[\Sp{10},{}^{>6}\btau_{\leq 10}~g/o]}
	\arrow[Rightarrow, no head, from=1-1, to=2-1]
	\arrow["{f^*}", from=1-1, to=1-2]
	\arrow["{q^*}", from=1-2, to=1-3]
	\arrow[Rightarrow, no head, from=1-2, to=2-2]
	\arrow["{f^*}", from=2-1, to=2-2]
    \arrow["{d^*}"',tail, from=2-3, to=2-2]
    \arrow["{d^*}", from=2-3, to=1-3]
\end{tikzcd}
\end{center}
whose top row is a part of exact sequence induced from the cofiber sequence $ \sfrac{\rp{10}}{\rp{6}}\overset{q}{\lr}\sfrac{\rp{10}}{\rp{7}}\overset{f}{\lr}\Sigma \sfrac{\rp{7}}{\rp{6}}$. Furthermore, the injectivity of the map $d^*:[\Sp{10},{}^{>6}\btau_{\leq 10}~g/o]_{(2)}\lr[\sfrac{\rp{10}}{\rp{7}},{}^{>6}\btau_{\leq 10}~g/o]_{(2)}$ can be verified using the exact sequence \eqref{ex seq: S10,E to rp10/7,E injective} by replacing $\Sigma^7\EE$ with ${}^{>6}\btau_{\leq 10}~g/o$.

Observe that, the connecting map $f:\sfrac{\rp{10}}{\rp{7}}\lr\Sigma \sfrac{\rp{7}}{\rp{6}}$ is homotopic to the map $ (2,\phi):\Sp{8}\vee M(\sfrac{\Z}{2},9) \lr \Sp{8}$. Here, the map $\phi$ fits in the following commutative triangle
\begin{center}
  \begin{tikzcd}
	{M(\sfrac{\Z}{2},9)} & {\Sp{8}} \\
	{\Sp{9}}
	\arrow["\phi", from=1-1, to=1-2]
	\arrow["\iota", hook, from=2-1, to=1-1]
	\arrow["\eta"', from=2-1, to=1-2]
\end{tikzcd}
\end{center}
According to \cite[Lemma 2.3]{HarmonicCrowley}, there is a non-split short exact sequence 
\begin{center}\label{sht seq: S10 to M(Z2,9) to S9 exact seq for S8}
    \begin{tikzcd}
	0 & {\underset{\sfrac{\Z}{2}=\lara{\eta^2}}{\underset{\rotatebox{90}{$ \cong$}}{[\Sp{10},\Sp{8}]}}} & {[M(\sfrac{\Z}{2},9),\Sp{8}]} & {\underset{\sfrac{\Z}{2}=\lara{\eta}}{\underset{\rotatebox{90}{$ \cong$}}{[\Sp{9},\Sp{8}]}}} & 0
	\arrow["d^*", from=1-2, to=1-3]
	\arrow[from=1-3, to=1-4]
	\arrow[from=1-4, to=1-5]
	\arrow[from=1-1, to=1-2]
\end{tikzcd}
\end{center}
which makes $\phi$ the generator of $[M(\sfrac{\Z}{2},9),\Sp{8}]=\sfrac{\Z}{4}$.
 In \eqref{sht seq: S10 to M(Z2,9) to S9 exact seq for S8}, observe that $d^*(\eta^2)=2\phi$. In the following commutative triangle, this forces the map $[\Sp{8},{}^{>6}\btau_{\leq 10}~g/o]\lr [M(\sfrac{\Z}{2},9),{}^{>6}\btau_{\leq 10}~g/o]$ to be $(2\phi)^*$
\begin{center}
  \begin{tikzcd}
	{{[\Sp{8},{}^{>6}\btau_{\leq 10}~g/o]}} \\
	{{[\Sp{10},{}^{>6}\btau_{\leq 10}~g/o]}} & {{[M(\sfrac{\Z}{2},9),{}^{>6}\btau_{\leq 10}~g/o]}}.
	\arrow["{(\eta^2)^*}"', from=1-1, to=2-1]
	\arrow["{d^*}"', from=2-1, to=2-2]
	\arrow[from=1-1, to=2-2]
\end{tikzcd}
\end{center}
The surjectivity of ${(\eta^2)}^*$ is surjective implies 
\begin{equation*}
    \begin{split}
        \im{d^*\colon[\Sp{10},{}^{>6}\btau_{\leq 10}~g/o]\to[M(\sfrac{\Z}{2},9),{}^{>6}\btau_{\leq 10}~g/o]} & \subseteq\im{(2\phi)^*} \\
        & \subseteq\im{\phi^*}\subseteq \im{(2,\phi)^*}.
    \end{split}
\end{equation*}

Now, it follows from the top row exact sequence in \eqref{dig: Rp7/6 to Rp10/6 to Rp10/7 at g/o} that \\
 $d^*:[\Sp{10},{}^{>6}\btau_{\leq 10}~g/o]\to [\sfrac{\rp{10}}{\rp{6}},{}^{>6}\btau_{\leq 10}~g/o]$ is trivial. Therefore, in \eqref{dig: degree one maps for rp10/7 to 10/6 to rp10 for g/o}, the map $d^*:[\Sp{10},{}^{>6}\btau_{\leq 10}~g/o]\to[\rp{10},{}^{>6}\btau_{\leq 10}~g/o]$ is trivial which completes the proof.

\end{proof}

We conclude this section with the following result using \Cref{thm: rp10}
\begin{cor}
  $I(\rp{10})=\Th{10}$.
\end{cor}

\end{document}